\documentclass[10pt,a4paper]{article}
\usepackage{amsmath,amssymb,mathrsfs,amsthm}
\usepackage{hyperref}
\usepackage{eucal}
\usepackage{graphicx}
\usepackage{color}
\usepackage{verbatim}
\newcommand{\Nat}{\mathbb{N}}
\newcommand{\Real}{{\mathbb{R}}}
\newcommand{\Com}{{\mathbb{C}}}

\newcommand{\T}{{\mathcal{T}}}
\newcommand{\dd}{{{\rm d}}}

\newcommand{\PT}{{\mathcal{PT}}}

\renewcommand{\H}{{\mathcal{H}}}
\newcommand{\sii}{L^2}
\newcommand{\Dom}{\mathop{\mathrm{Dom}}\nolimits}
\newcommand{\Ran}{\mathop{\mathrm{Ran}}\nolimits}
\newcommand{\Ker}{\mathop{\mathrm{Ker}}\nolimits}
\newcommand{\Num}{\mathop{\mathrm{Num}}\nolimits}
\newcommand{\cf}{\emph{cf.}}

\newcommand{\ie}{{\emph{i.e.}}}
\newcommand{\eg}{{\emph{e.g.}}}
\newcommand{\etal}{{\emph{et al.}}}
\newcommand{\dist}{\mathrm{dist}}

\newcommand{\eps}{\varepsilon}
\renewcommand{\Re}{\operatorname{Re}}
\renewcommand{\Im}{\operatorname{Im}}

\newcommand{\Hsa}{H_{\rm sa}}

\usepackage[normalem]{ulem}
\definecolor{DarkGreen}{rgb}{0,0.5,0.1}

\newtheorem{Theorem}{Theorem}
\newtheorem{Proposition}[Theorem]{Proposition}
\graphicspath{{C:/Data/00Synchronized/Figures/2013/PT_Pseudo/}}
\usepackage[update,prepend]{epstopdf} 
\begin{document}
\title{\bf Pseudospectra in non-Hermitian quantum mechanics}
\author{D.~Krej\v{c}i\v{r}\'{i}k$^{a}$, 
P.~Siegl$^{a,b}$, M.~Tater$^{a}$ 
and J.~Viola$^{c}$}
\date{
\small
\emph{
\begin{quote}
\begin{itemize}
\item[$a)$]
Department of Theoretical Physics, Nuclear Physics Institute ASCR, 
25068 \v{R}e\v{z}, Czech Republic.%
\\
\item[$b)$]
Mathematical Institute, University of Bern,
Alpeneggstrasse 22, 3012 Bern, Switzerland. 
\\
\item[$c)$]
D\'epartement de Math\'ematiques, 
Universit\'e de Nantes,
2 rue de la Houssini\`ere, BP 92208,
44322 Nantes Cedex 3, France. 
\end{itemize}
\end{quote}
}
\medskip
12th October 2015
}

\begin{center}
PREPRINT of \ \fbox{ J. Math. Phys. 56 (2015), 103513 }
\end{center}

\vspace{-5ex}

{\let\newpage\relax\maketitle}

\begin{abstract}
\noindent
We propose giving the mathematical concept of the pseudospectrum 
a central role in quantum mechanics 
with non-Hermitian operators. 
We relate pseudospectral properties to quasi-Hermiticity,
similarity to self-adjoint operators, 
and basis properties of eigenfunctions.
The abstract results are illustrated by unexpected wild
properties of operators familiar 
from $\mathcal{PT}$-symmetric quantum mechanics.
\end{abstract}

%
\section{Introduction}
%
\begin{quote}
\emph{In the highly non-normal case, vivid though the image may be,
the location of the eigenvalues may be as fragile an indicator 
of underlying character as the hair colour of a Hollywood actor.}

\raggedleft Trefethen and Embree, \cite[p.~11]{Trefethen-Embree}
\end{quote}
It has long been known to numerical analysts that many important spectral 
properties of self-adjoint operators 
are lost when considering non-normal operators.  
In the 2005 monograph \cite{Trefethen-Embree},
Trefethen and Embree 
discuss decades of ongoing research and advocate 
the use of pseudospectra instead of spectra 
when studying a matrix or operator which is 
non-self-adjoint or non-normal. 

In this paper we stress the importance of pseudospectra, a set which measures the instability of the spectrum of a non-normal operator,
in the so-called ``non-Hermitian quantum mechanics''. 
Contrary to a common misconception, we demonstrate that the spectrum alone
contains by far insufficient information to draw any
quantum-mechanically relevant conclusions 
for non-Hermitian operators.
In particular, the fact that the spectrum of an operator is real, 
or even the presence of a reduction to a self-adjoint
operator using an unbounded similarity transformation,
is not sufficient to guarantee that an operator 
with non-trivial pseudospectrum has a meaning in the context of  
conventional quantum mechanics.

By non-Hermitian quantum mechanics we mean the attempts
of Bender \etal\ \cite{Bender-Boettcher_1998,BBJ}
to extend quantum mechanics to include observables represented 
by $\mathcal{PT}$-symmetric non-Hermitian operators, see \cite{nsa-rem}, with real spectra.  
Here, $\mathcal{PT}$-symmetry refers to the invariance of 
an operator~$H$ on the Hilbert space $L^2(\Real^d)$
with respect to a simultaneous parity and time reversal, \ie,
\begin{equation}\label{symmetry}
  [H,\mathcal{PT}]=0 
  \,,
\end{equation}
where $(\mathcal{P}\psi)(x):=\psi(-x)$ 
and $(\mathcal{T}\psi)(x):=\overline{\psi(x)}$.
It has been argued that if the operator possesses,
in addition to the obvious $\mathcal{PT}$-symmetry,
a special hidden symmetry -- a so-called $\mathcal{C}$-symmetry --
then indeed the spectrum of~$H$ is real. 
It has furthermore been suggested that a consistent quantum
theory can be built by changing the inner product 
into one for which the operator~$H$ is Hermitian 
and the time evolution is unitary.
The procedure can be understood by the concept of pseudo-Hermiticity
as developed by Mostafazadeh \cite{Ali1,Ali2,Ali3}: 
a $\mathcal{CPT}$-symmetric operator can be transformed into a self-adjoint operator 
$\Hsa=\Hsa^*$ via a similarity transformation~$\Omega$, \ie,
\begin{equation}\label{similar}
  \Hsa = \Omega H \Omega^{-1}
  \,.
\end{equation}
The latter is the basis for a possible 
quantum-mechanical interpretation of~$H$ 
as an equivalent representation 
of a physical observable which would be conventionally
represented by the self-adjoint operator~$\Hsa$.

If~$\Omega$ in~\eqref{similar} is bounded and boundedly invertible,
\cf~\cite{bdd-rem},
although not necessarily unitary,
then indeed the spectra of~$\Hsa$ and~$H$ 
coincide and the pseudospectra  
are related by simple approximate inclusions; see \eqref{condition.number}.
Moreover, if the spectrum is discrete, 
the eigenfunctions of~$H$ and~$\Hsa$ 
share essential basis properties.
In this case, the $\mathcal{PT}$-symmetry can be understood
through an older notion of quasi-Hermiticity~\cite{Dieudonne_1961,GHS}
and the quantum-mechanical description of~$H$ via~$\Hsa$ is consistent:
$H$ and $\Hsa$ represent the same physical system.

However, problems arise if~$\Omega$ or~$\Omega^{-1}$ 
entering the fundamental relation \eqref{similar} are allowed to be unbounded.
We list several potential pitfalls below. 
We do not claim that there are no physical problems 
where an unbounded similarity transformation could be useful (in fact, there are!), 
nevertheless, if any of the pathological situations described below occur, 
$\Hsa$ and~$H$~cannot be viewed as equivalent 
representatives of the same physical observable in quantum mechanics.

\begin{enumerate}
\item
It is not always easy to give a good meaning 
to the operator identity~\eqref{similar} 
when taking into account the respective domains. 
The relation \eqref{similar} may hold on 
some particular functions, 
\eg~$C_0^\infty(\Real^d)$, 
but the operator identity may not be satisfied. 
\item
Spectra are not preserved by unbounded transformations.
It may well happen that the spectrum of~$H$ is purely discrete,
while~$\Hsa$ has no eigenvalues or some continuous spectrum, 
and vice versa. 
\item
Unbounded transformations may turn a nice (even orthonormal) eigenbasis
to a set of functions that cannot form any kind
of reasonable basis.
\item
Spectra of non-Hermitian operators 
are known to lie deep inside very large pseudospectra, 
while the pseudospectrum of a self-adjoint operator 
is just a tubular neighbourhood of its spectrum.   
Consequently, the spectrum of $\Hsa$~is stable under
small perturbations, while an arbitrarily small perturbation of~$H$
can create eigenvalues very far from the spectrum of~$H$.
\item 
Even if the spectrum of~$H$ were purely real, 
$-iH$ does not need to be the generator of a bounded semigroup.
In fact, a wild behaviour of the pseudospectrum of~$H$ 
prevents to associate a bounded time-evolution to~$H$ 
via the Schr\"odinger equation (\cf~\cite[Thm.~8.2.1]{Davies_2007}).  
\end{enumerate}

The objective of this paper is to demonstrate 
by a careful mathematical analysis that
these commonly overlooked problems do appear in more or less 
famous models of non-Hermitian quantum mechanics 
and to emphasize that the concept of pseudospectra 
gives important information missing in prior works. 
In conclusion, the present study necessarily
casts doubt on certain commonly accepted conclusions
based on formal manipulations in the physical literature, 
\cf~the reviews \cite{Bender_2007,Ali-review},
particularly on the physical relevance of $\mathcal{PT}$-symmetry
in the quantum-mechanical context.
We notably refer to the concrete examples 
presented in Section~\ref{Sec.examples} for specific controversies.
We also remark that the unbounded time-evolution has been recently 
studied more explicitly for some of the models presented below,
namely in~\cite{Graefe-2015-48} for the gauged oscillator
(see Section~\ref{subsec.Swanson})
and in~\cite{Aleman-2014a,Aleman-2014b} for quadratic operators 
(see Section~\ref{subsec.el.q.op}).

Our approach relies on standard tools of modern functional analysis, nonetheless, several innovations appear.
Particularly, we construct new pseudomodes for the (non-semiclassical) shifted harmonic oscillator, where the scaling leads to the fractional power of $h$, \cf~Theorem \ref{Thm.Shifted}. 
We also
give a new and very short proof of the rate of spectral projection growth for the rotated oscillator analysed previously in \cite{Davies-Kuijlaars_2004,Henry,Viola_2013}, \cf~Appendix \ref{app.Pk}. 
Last but not least, we establish an explicit unitary equivalence 
between the rotated and gauged oscillator (Swanson's model),
which has not been noticed previously. We furthermore describe how to identify an equivalent rotated oscillator for each such quadratic operator in a class identified in \cite{Pravda-Starov_2008}.

The paper is organized as follows.
In the Section~\ref{Sec.pseudo}
we summarize some basic properties of pseudospectra.
Since~$\mathcal{PT}$-symmetry is a special example of an antilinear symmetry,
we briefly discuss
pseudospectral properties of operators 
commuting with antiunitary operators in Section~\ref{Sec.anti}.
Our main mathematical tool for proving the existence of large pseudospectra
is stated as Theorem~\ref{Thm.DSZ} of Section~\ref{Sec.micro};
it is based on a construction of pseudomodes of semiclassical operators
adapted from~\cite{Dencker-Sjostrand-Zworski_2004}.
In Section~\ref{Sec.metric} we relate the pseudospectrum
to the concept of quasi-Hermiticity  
and similarity to a self-adjoint operator.
A relationship between basis properties of eigenfunctions 
is pointed out in Section~\ref{Sec.basis}.
Finally, in Section~\ref{Sec.examples} (which occupies the bulk of the paper), we present a number
of non-self-adjoint ordinary differential operators 
exhibiting striking spectral and pseudospectral properties;
they will serve as an illustration of the abstract 
operator-theoretic methods summarized in this paper.  Certain technical proofs are reserved for the Appendix.

\section{Pseudospectra}\label{Sec.pseudo}
%
The notion of pseudospectra arose as a result of the realization 
that several pathological properties of highly non-Hermitian 
operators were closely related. 
We refer to by now classical monographs 
by Trefethen and Embree~\cite{Trefethen-Embree}
and Davies~\cite{Davies_2007} for more information
on the subject and many references.

Let~$H$ be a closed densely defined operator (bounded or unbounded)
on a complex Hilbert space~$\mathcal{H}$.
The \emph{spectrum} of~$H$, denoted by~$\sigma(H)$,  
consists of those complex points~$z$
for which the resolvent $(H-z)^{-1}$ does not exist
as a bounded operator on~$\mathcal{H}$.
If~$\H$ were finite-dimensional, then the spectrum of~$H$
would be exhausted by eigenvalues
(\ie~those complex numbers~$\lambda$  
for which $H-\lambda$ is not injective).
In general, however, there are additional parts of the spectrum composed  
of those~$\lambda$ which are not eigenvalues 
but for which $H-\lambda:\Dom(H)\to\H$ is not bijective:
depending on whether the range $\Ran(H-\lambda)$ 
is dense in~$\H$ or not, one speaks about the \emph{continuous}
or \emph{residual} spectrum, respectively.

The complement of~$\sigma(H)$ in~$\Com$
is called the \emph{resolvent set} of~$H$.
The \emph{numerical range}~$\Num(H)$ of~$H$ is defined by
the set of all complex numbers $(\psi,H\psi)$
where~$\psi$ ranges over all~$\psi$ from
the operator domain $\Dom(H)$ with $\|\psi\|=1$.  

Given a positive number~$\eps$,
we define the \emph{$\varepsilon$-pseudospectrum} (or simply \emph{pseudospectrum}) of~$H$ 
as
\begin{equation}\label{pseudo}
  \sigma_\eps(H) := \sigma(H) \cup
  \big\{
  z \in \Com \ \big| \ 
  \|(H-z)^{-1}\| > \eps^{-1} 
  \big\}
  \,;
\end{equation}
sometimes the convention that 
$\|(H-z)^{-1}\|=\infty$ for $z \in \sigma(H)$ is used.
(We refer to the interesting essays 
\cite{Shargorodsky_2009,Shargorodsky_2010} by Shargorodsky
on the distinction between the definition of pseudospectra 
with strict and non-strict inequalities.) 
Here we summarize some basic well-known properties of pseudospectra.

\paragraph{$\bullet$ Topology.}
For every $\eps>0$, 
$\sigma_\eps(H)$ is a non-empty open subset of~$\Com$
and any bounded connected component of $\sigma_\eps(H)$
has a non-empty intersection with $\sigma(H)$.
(If the spectrum of~$H$ is empty, then $\sigma_\eps(H)$
is unbounded for every $\eps>0$.) 
These facts follow from the subharmonicity of $\|(H-z)^{-1}\|$
as a function of $z \not\in \sigma(H)$ to $(0,\infty)$.

\paragraph{$\bullet$ Relation to spectra.}
The pseudospectrum always contains an $\eps$-neighbour\-hood 
of the spectrum, and if $\Com\setminus\overline{\Num(H)}$ 
is connected and has a non-empty intersection 
with the resolvent set of~$H$,
the pseudospectrum is in turn contained in an $\eps$-neighbourhood 
of the numerical range:
\begin{equation}\label{robust}
  \big\{z\in\Com\ \big| \ \dist\big(z,\sigma(H)\big) < \eps \big\}
  \subseteq \sigma_\eps(H) \subseteq 
  \big\{z\in\Com\ \big| \ \dist\big(z,\overline{\Num(H)}\big) < \eps \big\}
  \,.
\end{equation}
The first inclusion follows from the bound
$\|(H-z)^{-1}\| \geq \dist\big(z,\sigma(H)\big)^{-1}$, which is
valid for any operator.
Since equality holds there if~$H$ is self-adjoint 
(or more generally normal, \ie\ $H^*H=HH^*$), 
it follows that 
the pseudospectra for such operators give no additional information not already given by the spectrum.
On the other hand, if~$H$ is ``highly non-self-adjoint'',
the pseudospectrum $\sigma_\eps(H)$ is typically ``much larger''
than the $\eps$-neighbourhood of the spectrum. 
In any case, the second inclusion shows 
that the pseudospectra are well behaved outside the numerical range.

\paragraph{$\bullet$ Spectral instability.}
There is an important property
(known sometimes as the Roch-Silberman theorem \cite{Roch-Silberman_1996},
although the result is already mentioned in \cite{Reichel-1992-162})
relating the pseudospectra to the stability
of the spectrum under small perturbations:
\begin{equation}\label{RS-thm}
  \sigma_\eps(H) = 
  \bigcup_{\|V\| < \eps} \sigma(H+V)
  \,.
\end{equation}
The importance of this property is summarized by the following statement 
from~\cite{Davies_2002}:
\begin{quote}\emph{Very large pseudospectra are always associated with eigenvalues 
which are very unstable with respect to perturbations. 
This is clearly of great importance to numerical analysts: 
if a spectral problem is unstable enough, no numerical procedure can
enable one to find the eigenvalues, 
whose significance therefore becomes a moot point.}
\end{quote}
The relation~\eqref{RS-thm} likely lends the strongest support
for the usage of pseudospectra instead of spectra
in the case of non-Hermitian operators.

\paragraph{$\bullet$ Pseudomodes.}
A complex number~$z$ belongs to $\sigma_\eps(H)$ 
if, and only if, $z\in\sigma(H)$ or~$z$ is 
a \emph{pseudoeigenvalue} (or \emph{approximate eigenvalue}), \ie,  
\begin{equation}\label{pseudoEV}
  \|(H-z)\psi\| < \eps \|\psi\|
  \quad \mbox{for some} \quad \psi \in \Dom(H)
  \,.
\end{equation}
Any~$\psi$ satisfying~\eqref{pseudoEV} 
is called a \emph{pseudoeigenvector}
(or \emph{pseudoeigenfunction} or \emph{pseudomode}).
Again, for operators~$H$ which are far from self-adjoint, 
pseudoeigenvalues may not be close to the spectrum of~$H$.
This is particularly striking if we realize that
these pseudoeigenvalues can be turned into true eigenvalues
by a very small perturbation, \cf~\eqref{RS-thm}. 
What is more, we can often construct very nice 
(\eg~smooth and with compact support) pseudoeigenfunctions; 
see Section~\ref{Sec.micro} and the references therein.

\paragraph{$\bullet$ Adjoints.}
Using the identity 
$
  \|(H^*-\bar{z})^{-1}\| = \|(H-z)^{-1}\|
$,
it is easy to see that the pseudospectrum of~$H^*$
is given by the mirror image of $\sigma_\eps(H)$
with respect to the real axis, \ie,
\begin{equation}\label{adjoint}
  \lambda \in \sigma_\eps(H)
  \ \Longleftrightarrow \
  \overline{\lambda} \in \sigma_\eps(H^*)
    \,.
\end{equation}

\paragraph{$\bullet$ Similarity.}
Let the similarity relation~\eqref{similar} 
hold with a bounded and boundedly invertible 
operator~$\Omega$. Then the operators~$H$ and~$\Hsa$
have the same spectra, but their pseudospectra
may be very different, unless the \emph{condition number}
$\kappa:=\|\Omega\| \|\Omega^{-1}\| \geq 1$ 
is fairly close to one.  
Indeed, we have
\begin{equation}\label{condition.number}
  \sigma_{\eps/\kappa}(H) 
  \subseteq \sigma_{\eps}(\Hsa) \subseteq
  \sigma_{\eps\kappa}(H) 
  \,.
\end{equation}
%

\section{Antilinear symmetry}\label{Sec.anti}
%
We understand $\mathcal{PT}$-symmetry as a special example
of invariance of a closed densely defined operator~$H$ 
with respect to an \emph{antiunitary} transformation~$\mathcal{S}$, \ie,
\begin{equation}\label{symmetry.bis}
  [H,\mathcal{S}] = 0
  \,.
\end{equation}
Recall that $\mathcal{S}$ is an antiunitary
transformation if
$\mathcal{S}$~is a bijective antilinear operator on~$\mathcal{H}$
satisfying $(\mathcal{S}\phi,\mathcal{S}\psi)=(\psi,\phi)$ 
for every $\phi,\psi \in \mathcal{H}$. 
As usual for the commutativity of an unbounded operator
and a bounded operator \cite[Sec.~III.5.6]{Kato}, 
we understand~\eqref{symmetry.bis}
by the operator relation $\mathcal{S}H \subseteq H\mathcal{S}$.
In other words, whenever $\psi\in\Dom(H)$, 
the image $\mathcal{S}\psi$~also belongs to~$\Dom(H)$ 
and $\mathcal{S}H\psi = H\mathcal{S}\psi$.

\paragraph{$\bullet$ Symmetry.}
It is well known that the spectra of $\mathcal{PT}$-symmetric
operators on $L^2(\Real^d)$ are symmetric with respect to the real axis.
In our general situation~\eqref{symmetry.bis}, 
this follows from the identity
\begin{equation}\label{antilinear}
  (H-z)^{-1} = \mathcal{S}^{-1}(H-\bar{z})^{-1}  \mathcal{S},
\end{equation}
which is valid for any~$z$ in the resolvent set of~$H$;
the symmetry can then be deduced from~\eqref{symmetry.bis}. 
Furthermore, \eqref{antilinear}~yields the same relation
for the pseudospectra of $\mathcal{S}$-symmetric operators~$H$:
\begin{equation}\label{antilinear.pseudo}
  \lambda \in \sigma_\eps(H)
  \ \Longleftrightarrow \
  \overline{\lambda} \in \sigma_\eps(H)
  \,.
\end{equation}
This identity holds trivially when the resolvent set 
of~$H$ is empty.

\paragraph{$\bullet$ $\mathcal{J}$-self-adjointness.}
An alternative framework for $\mathcal{PT}$-symmetric operators
was suggested in~\cite{BK} 
in terms of $\mathcal{J}$-self-adjoint operators.
Here~$\mathcal{J}$ is a \emph{conjugation},
\ie~an antiunitary involution,
and~$H$ is said \emph{$\mathcal{J}$-self-adjoint} if 
$H^*=\mathcal{J}H\mathcal{J}$,
\cf~\cite[Sec.~III.5]{Edmunds-Evans}.
(The present $\mathcal{J}$-self-adjointness 
should not be confused with
a terminologically similar but different concept 
in Krein spaces~\cite{Albeverio-Guenther-Kuzhel_2009,Hassi-Kuzhel_2013}.)
An example of such a~$\mathcal{J}$ on $\sii(\Real^d)$ is $\T$, \ie~complex conjugation.
A particularly useful property of $\mathcal{J}$-self-adjoint
operators is that their residual spectrum is always empty, \cf~\cite{BK}.

\section{Microlocal analysis}\label{Sec.micro}
%
It was Davies~\cite{Davies_1999-NSA} who first realized 
that and how semiclassical methods can be applied to the study 
of pseudospectra of non-Hermitian Schr\"odinger operators.
Shortly thereafter, Zworski~\cite{Zworski_2001} 
pointed out that Davies' discoveries
could be related to long-established results 
in the microlocal theory of partial differential operators
due to H\"ormander and others.
We refer to a paper of Dencker, Sj\"ostrand and Zworski
\cite{Dencker-Sjostrand-Zworski_2004} for an important
development of the idea in the context of pseudodifferential operators.
We state here a version of their general result for the special case 
of differential operators with analytic coefficients in one dimension, in a formulation given in \cite[Thm.~11.1]{Trefethen-Embree}.

We recall some standard notions first. 
Let $h>0$ be a (small) parameter. 
Given continuous functions 
$a_j:\Real\to\Com$, with $j=0,\dots,n$, 
we define a \emph{symbol}
\begin{equation}\label{sym.def}
f(x,\xi) := \sum_{j=0}^n a_j(x) (-i\xi)^j
\,, \qquad
(x,\xi) \in \Real^2
\,,
\end{equation}
and the associated
\emph{semiclassical differential operator}
\begin{equation}\label{Hh.def}
	H_h := \sum_{j=0}^n a_j(x) \, 
	h^j \, \frac{\dd^j}{\dd x^j} 
	\,, 
	\qquad \Dom(H_h):=C_0^{\infty}(\Real)
	\,.
\end{equation}
The \emph{Poisson bracket}
$\{\cdot,\cdot\}$ is defined as
\begin{equation}\label{poisson}
	\{u,v\} := 
	\frac{\partial u}{\partial \xi} \frac{\partial v}{\partial x}
	- \frac{\partial u}{\partial x} \frac{\partial v}{\partial \xi}
\end{equation}
and, for $u=f$, $v=\bar{f}$, it simplifies to
$$
\{f,\bar{f}\} = 2 i \left(
\frac{\partial \Im f}{\partial \xi} \frac{\partial \Re f}{\partial x}
- \frac{\partial \Re f}{\partial \xi} \frac{\partial \Im f}{\partial x}
\right).
$$
The closure of the set
\begin{equation}\label{semiclass.pseudospec}
	\Lambda :=
	\Big\{
	f(x,\xi) \, : \,
	(x,\xi) \in \Real^{2}, \ \frac{1}{2i}\{f,\bar{f}\}(x,\xi) > 0 
	\Big\}
\end{equation}
is referred to as 
the \emph{semiclassical pseudospectrum}\index{semiclassical pseudospectrum} 
of $H_h$, \cf~\cite{Dencker-Sjostrand-Zworski_2004}.
In the special case of~$H_h$ 
being a Schr\"odinger operator with an analytic potential, 
the condition $\frac{1}{2i}\{f,\bar{f}\}(x,\xi) > 0$ 
reduces to $\Im V'(x) \neq 0$ and $\xi\not = 0$, because the sign of $\xi$ can be chosen freely, and it is also equivalent to the \emph{twist condition} of~\cite[Sec.~III.11]{Trefethen-Embree}.
The nonvanishing of $\{f,\bar{f}\}$ 
is a classical analogue of the operator~$H_h$ 
not being normal~\cite{Dencker-Sjostrand-Zworski_2004}.

\begin{Theorem}[Semiclassical pseudomodes.]
\label{Thm.DSZ}
Let the functions $a_j$, $j=0,\dots,n$, be analytic and let $H_h$ be the semiclassical differential operator \eqref{Hh.def}. 
Then for any~$z \in \Lambda$,  
there exist $C = C(z) > 1$, $h_0 =h_0(z) > 0$ 
and an $h$-dependent family of $C_0^\infty(\Real)$ 
functions $\{\psi_h \}_{0 < h \leq h_0}$
with the property that, for all $0 < h \leq h_0$,
$$
\|(H_h-z)\psi_h \| < C^{-1/h} \, \|\psi_h\|.
$$
\end{Theorem}

Such of family of functions is called a pseudoeigenfunction 
(or \emph{pseudomode}) 
for the operator~$H_h$ 
corresponding to \emph{pseudoeigenvalue}
(or \emph{approximate eigenvalue})~$z$. 

In Appendix~\ref{App.proofs}, we include a proof 
of the theorem
in the special case of Schr\"odinger operators.

The theorem can be generalized significantly beyond 
the restrictive assumption that the coefficients are globally analytic.  As stated in \cite[Thm.~11.1]{Trefethen-Embree}, one only needs that the $a_j$ are analytic in a neighborhood of some $x_0 \in \Real$ corresponding to an $(x_0, \xi_0) \in \Real^2$ putting $z \in \Lambda$. Furthermore, \cite{Dencker-Sjostrand-Zworski_2004} 
shows the existence of pseudomodes for pseudodifferential operators 
(which include differential operators) 
whose symbols are only assumed to be smooth (and bounded).  
The price of abandoning the assumption of analytic coefficients 
is a slower rate of growth. 
In the smooth case, instead of an upper bound of $C^{-1/h} \|\psi_h\|$ 
one has an upper bound for each $N \in \Nat$ and constant $C(N) > 0$ 
depending on $N$:
\[
	\|(H_h - z)\psi_h\| < \frac{h^N}{C(N)} \, \|\psi_h\| \,,
\]
for all $0 < h \leq h_0$.
For the examples to follow, it will be sufficient to consider the analytic case as written in Theorem~\ref{Thm.DSZ}.

Although Theorem~\ref{Thm.DSZ} is stated for semiclassical operators, scaling techniques allow its application to non-semiclassical operators where the spectral parameter tends to infinity.  This is based on the principle that the semiclassical limit is equivalent to the high-energy limit after a change of variables; this principle is made concrete in many of the examples below. For further details, the reader could consult \cite{Zworski_2012}.

\section{Metric operators}\label{Sec.metric}
%
The similarity relation~\eqref{similar} is closely related 
to the \emph{quasi-Hermiticity} of~$H$:
\begin{equation}\label{quasi}
  H^*\Theta = \Theta H 
  \,,
\end{equation}
where~$\Theta$ is a positive operator called 
a \emph{metric} \cite{Dieudonne_1961,GHS}.
The terminology comes from the observation that~$H$
is formally self-adjoint with respect to the modified
inner product $\langle\cdot,\Theta\cdot\rangle$.
More precisely, we have
\begin{Proposition}\label{Prop.equivalence}
The operator $H$~is similar to a self-adjoint operator via a bounded
and boundedly invertible positive transformation (\ie~\eqref{similar} holds)
if, and only if, $H$~is quasi-Hermitian with a positive,
bounded, and boundedly invertible metric (\ie~\eqref{quasi} holds).  
\end{Proposition}
\begin{proof}
If~$H$ satisfies~\eqref{quasi} with a bounded
and boundedly invertible $\Theta>0$, then the similarity
relation~\eqref{similar} to a self-adjoint operator~$h$
holds with any~$\Omega$ satisfying the decomposition $\Theta=\Omega^*\Omega$.  
Conversely, if~\eqref{similar} holds with a bounded
and boundedly invertible~$\Omega$, then it is easy to check
that~\eqref{quasi} holds with $\Theta = \Omega^*\Omega>0$. 
\end{proof}
As emphasized in~\cite{SK},
fundamental problems arise if one starts to relax the conditions 
on boundedness or bounded invertibility of the transformations. 

\paragraph{$\bullet$ Trivial pseudospectra.}
We say that the pseudospectrum of~$H$ is \emph{trivial}
if there exists a fixed constant~$C$ such that, for all $\eps > 0$,
$$
  \sigma_\eps(H)
  \subseteq
  \big\{z\in\Com\ \big| \ \dist\big(z,\sigma(H)\big) < C\eps \big\}
  \,.
$$
That is, the pseudospectrum of~$H$ 
is contained in a tubular neighbourhood
of the spectrum of~$H$ (although of possibly larger radius than~$\eps$).
Recalling the text below~\eqref{robust}, 
the pseudospectra of self-adjoint and normal operators are trivial.
\begin{Proposition}\label{Prop.test}
Let~$H$ be quasi-Hermitian~\eqref{quasi} with a positive,
bounded and boundedly invertible metric. 
Then the pseudospectrum of~$H$ is trivial.
\end{Proposition}
\begin{proof}
It is enough to recall Proposition~\ref{Prop.equivalence}
and~\eqref{condition.number}; the condition number plays
the role of the constant~$C$.
\end{proof}

Proposition~\ref{Prop.test} can be conveniently used 
in the reverse sense, where the presence of nontrivial pseudospectrum for a given operator~$H$ immediately implies that the operator cannot possess a physically relevant
(\ie~bounded and boundedly invertible) metric.

\section{Basis properties}\label{Sec.basis}
%
Let~$H$ be an operator with compact resolvent
throughout this section.
Then the spectrum of~$H$ consists entirely
of isolated eigenvalues with finite (algebraic) multiplicities. 
Below, we recall that similarity to a self-adjoint operator, or quasi-Hermiticity, is equivalent to having the set of eigenvectors form a Riesz basis.

First, we recall the definition of a basis.
We say that $\{\psi_k\}_{k=1}^\infty$ is a 
\emph{(Schauder or conditional) basis} 
if every $\psi \in \H$ has a unique expansion in the vectors $\{\psi_k\}$, \ie~
\begin{equation}\label{basis}
  \forall \psi \in \H,
  \quad
  \exists! \{\alpha_k\}_{k=1}^\infty,
  \quad
  \psi = \sum_{k=1}^\infty \alpha_k \psi_k
  \,,
\end{equation}
where the infinite sum is understood as a limit 
in the strong topology of~$\H$.

\paragraph{$\bullet$ Riesz basis.}
We say that $\{\psi_k\}_{k=1}^\infty$, 
normalized to~$1$ in~$\H$, 
forms a \emph{Riesz (or unconditional) basis} if it forms a basis and the inequality 
\begin{equation}\label{Riesz}
\forall \psi \in \H, \qquad
  C^{-1} \|\psi\|^2
  \leq \sum_{k=1}^\infty |\langle\psi_k,\psi\rangle|^2 \leq 
  C \|\psi\|^2
\end{equation}
holds with a positive constant $C$ independent of $\psi$ (see \cite[Thm.3.4.5]{Davies_2007} for equivalent formulations).
In view of the Parseval equality, for any self-adjoint
operator with a purely discrete spectrum one may choose orthonormal eigenvectors which form a Riesz basis with $C = 1$.
For non-self-adjoint operators, however, it is not even clear
that the eigenfunctions form a basis
or even a complete set in~$\mathcal{H}$.
\begin{Proposition}\label{Prop.Riesz}
Let~$H$ be an operator with compact resolvent for which $\sigma(H) \subset \Real$.
Then $H$~is quasi-Hermitian with a positive,
bounded and boundedly invertible metric (\ie~\eqref{quasi} holds)
if, and only if, the eigenfunctions of~$H$ form a Riesz basis
(\ie~\eqref{Riesz} holds).  
\end{Proposition}
\begin{proof}
If~$H$ is quasi-Hermitian with bounded, boundedly invertible, and positive metric $\Theta$,
then, by Proposition~\ref{Prop.equivalence},
$H$~is similar to a self-adjoint operator~$h$  
via a bounded and boundedly invertible transformation~$\Omega$.
Consequently, $\Omega\psi_k/\|\Omega\psi_k\|$ 
form a complete orthonormal family in~$\mathcal{H}$
satisfying the Parseval inequality, from which~\eqref{Riesz} follows.
Conversely, assuming~\eqref{Riesz}, we construct a bounded
and boundedly invertible positive operator
$
  L := \sum_{k=1}^\infty \psi_k \langle\psi_k,\cdot\rangle
$.
It is easy to check that~\eqref{quasi} holds with~$\Theta=L^{-1}$.
\end{proof}

Combining Proposition~\ref{Prop.Riesz} with Proposition~\ref{Prop.test},
we see that the pseudospectrum can be employed as a useful indicator of
whether a non-self-adjoint operator possesses a Riesz basis.  

\paragraph{$\bullet$ No basis.}
Eigensystems of non-self-adjoint operators can have very wild basis properties. 
We recall that if $\{\psi_k\}_{k=1}^\infty$ is a basis, 
then there exists a sequence $\{ \phi_k\}_{k=1}^\infty$ for which the pair $\{\psi_k\}, \{\phi_k\}$ is biorthogonal, 
\ie~$\langle \phi_m,\psi_n\rangle = \delta_{m,n}$, such that 
$\alpha_k=\langle \phi_k, \psi \rangle$, \cf~\cite[Lem.~3.3.1]{Davies_2007}.
Let us denote the associated one-dimensional projections as 
\begin{equation}\label{biorthogonal.Pk}
	P_k := \psi_k\langle\phi_k,\cdot\rangle.
\end{equation}
The uniform boundedness principle is used to derive the following standard result.

\begin{Proposition}\label{Prop.basis}
If $\{\psi_k\}_{k=1}^\infty$ is a basis, then both $P_k$ and $\sum_{k=1}^N P_k$ are uniformly bounded in $\H$.
\end{Proposition}
\begin{proof}
The definition \eqref{basis}~implies that, for every $\psi \in \H$,
\begin{equation}
\forall \varepsilon >0, 
\ \exists N_{\psi,\eps}, 
\ \forall n,m > N_{\psi,\eps}, 
\quad
\left\|
\sum_{k=1}^n P_k \psi - \sum_{k=1}^m P_k \psi
\right\| 
< \varepsilon
\,.
\end{equation}
In particular, putting $m:=n-1$ and $\eps:=1$, 
we obtain $\|P_n \psi\| < 1$ if $n>N_{\psi,\eps}$.
Consequently, $\sup_k \|P_k\psi\| < \infty$ for every $\psi \in \H$.
Finally, the uniform boundedness principle~\cite[Thm.~III.9]{RS1}
yields $\sup_k \|P_k\| < \infty$. 
The proof of the second claim is analogous, see also \cite[Lem.~3.3.3]{Davies_2007}.
\end{proof}

For operators~$H$ with positive discrete spectrum, 
the basis property of the eigensystem of~$H$ may be excluded 
by the following corollary of \cite[Thm.~3]{Davies_2000}.
\begin{Proposition}\label{Prop.basis.D}
Let $H$ be an operator with compact resolvent, let its eigenvalues be simple and satisfy $\lambda_n \geq b n^\beta$ for some $b, \beta>0$ and all $n \in \Nat$ and let the corresponding eigenvectors form a basis.  Then there exist positive constants $k,m$ such that
\begin{equation}\label{R.Dav}
  \|(H-z)^{-1}\| \leq k \frac{(1+|z|^2)^\frac m2}{|\Im z|}
  \,, 
  \qquad z \notin \Real.
\end{equation}
\end{Proposition}

In Section~\ref{Sec.examples} we shall give a number of examples
of non-self-adjoint operators which have 
no basis of (generalized) eigenvectors.
This will immediately follow upon demonstrating 
that the norms of their eigenprojectors diverge 
or that the pseudospectrum of~$H$ does not obey the restriction \eqref{R.Dav}.
The resolution of the identity, 
which plays a central role in quantum mechanics,
is simply not available.

\paragraph{$\bullet$ Completeness.}
There exist weaker notions of basis in the literature
(\eg\ Abel-Lidskii basis),
but they are not nearly as useful in applications.
The weakest is to merely require 
the \emph{completeness} of~$\{\psi_k\}_{k=1}^\infty$,
\ie~that the span of $\{\psi_k\}_{k=1}^\infty$ is dense in~$\H$, 
or equivalently $\big({\rm span} \{\psi_k\}_{k=1}^\infty)^{\perp} = \{0\}$.
We have a converse of Proposition~\ref{Prop.basis} for a minimal complete set $\{\psi_k\}_{k=1}^\infty$, which implies that the projections $P_k$ in \eqref{biorthogonal.Pk} may be defined:
if the sums of projectors~$\sum_{k=1}^N P_k$ are uniformly bounded, 
then $\{\psi_k\}_{k=1}^\infty$ is a basis,
\cf~\cite[Lem.~3.3.3]{Davies_2007}.

\section{Examples}\label{Sec.examples}
%
In this last section we present a number of examples exhibiting
remarkable pseudospectral behaviour due to non-Hermiticity
and use them to demonstrate that the concept of pseudospectrum is a more
relevant consideration for the description of the operators,
specifically in the context of quantum mechanics.

We restrict ourselves mainly to the concrete situation of 
one-dimensional differential operators
familiar from ``non-Hermitian quantum mechanics''.
However, we emphasize that there exist versions of Theorem~\ref{Thm.DSZ}
for partial differential (even pseudodifferential) operators too
(see~\cite{Zworski_2001,Dencker-Sjostrand-Zworski_2004})
and it is straightforward to construct similar examples
with non-trivial pseudospectra in higher dimensions.
Non-Hermitian spectral effects for $\mathcal{PT}$-symmetric waveguides
were previously observed in~\cite{BK,KT,BK2}.

To avoid complicated notation, 
we study an operator~$H$ 
which changes in each subsection.  
Where there is a parameter dependence, we may write a subscript as in~$H_h$. 
The notation~$\Hsa$ will denote a self-adjoint operator, 
related to~$H$ usually via a formal (unbounded) conjugation.
The symbol~$C$ (occasionally with a subscript) will denote
a generic constant which may change from line to line.

\subsection{The imaginary Airy operator}
\label{subsec.Airy}
The non-self-adjoint operator 
\begin{equation}\label{Airy}
  H := - \frac{\dd^2}{\dd x^2} + ix
  \qquad \mbox{on} \qquad
  \sii(\Real)
\end{equation}
arises in the Ginzburg-Landau model of superconductivity
\cite{Almog_2008,Almog-Helffer-Pan_2013,Almog-Helffer-Pan_2012}
and also in the study of resonances of quantum Hamiltonians
with electric field via the method of complex scaling~\cite{Herbst-1979-64}. 
It is well defined as a closed operator when
considered on its maximal domain
\begin{equation}\label{Airy.dom}
  \Dom(H) := \big\{
  \psi \in \sii(\Real) \ | \, -\psi''+ix\psi \in \sii(\Real)
  \big\}
  \,.
\end{equation}
Indeed, such a definition coincides with the closure of~\eqref{Airy}
initially defined on smooth functions of compact support,
\cf~\cite[Cor.~VII.2.7]{Edmunds-Evans}.
More importantly, $H$~is $m$-accretive, 
\ie,\ the numerical range $\Num(H)$ is contained in 
the closed right complex half-plane and the resolvent bound 
$
  \|(H-z)^{-1}\| \leq |\Re z|^{-1} 
$
holds for all~$z$ with $\Re z < 0$.
The adjoint~$H^*$ of~$H$ is simply obtained 
by replacing $i$ with $-i$ in~\eqref{Airy} and~\eqref{Airy.dom}.	
Furthermore, $H$~is $\mathcal{PT}$-symmetric and $\mathcal{T}$-self-adjoint.  

\paragraph{$\bullet$ Spectrum.}
Integrating by parts, we easily check that 
\begin{align}\label{for.advection}
  \|\psi'\|^2 &= \langle\psi',\psi'\rangle 
  = \langle\psi,-\psi''\rangle
  \leq \|\psi\| \|\psi''\|
  \leq \delta \|\psi''\|^2 + \delta^{-1} \|\psi\|^2
  \,,
  \\ \nonumber
  \|H\psi\|^2 &= \|\psi''\|^2 + \|x\psi\|^2 + 2\Re\langle ix\psi,-\psi''\rangle
  = \|\psi''\|^2 + \|x\psi\|^2 + 2\Re\langle i\psi,\psi'\rangle
  \\ \nonumber
  &\geq \|\psi''\|^2 + \|x\psi\|^2 - 2 \|\psi\| \|\psi'\|
  \geq \|\psi''\|^2 + \|x\psi\|^2 - \delta \|\psi'\|^2 - \delta^{-1} \|\psi\|^2
  \,,
\end{align}
for every $\psi \in C_0^\infty(\Real)$ and $\delta > 0$.
Combining these inequalities for $\delta > 0$ sufficiently small
and using the density of $C_0^\infty(\Real)$ in $\Dom(H)$,
we arrive at the non-trivial fact that
$$
  \Dom(H) = \big\{
  \psi \in W^{2,2}(\Real) \ | \ x\psi \in \sii(\Real)
  \big\}
  \,.
$$
Here $W^{2,2}(\Real)$ denotes the usual Sobolev space of functions 
in $\sii(\Real)$ whose weak first and second derivatives
belong to $\sii(\Real)$, \cf~\cite{Adams}. 
Now it is clear that $\Dom(H)$ is
compactly embedded in $\sii(\Real)$ and~$H$ is an operator
with compact resolvent, \cf~\cite[Thm.~XIII.65]{RS4}. 
It follows that the spectrum of~$H$ 
may consist of isolated eigenvalues only. 
However, the eigenvalue equation $H\psi=\lambda\psi$ implies
that for any $c \in \Real$ we also have $H\psi_c=\lambda_c\psi_c$ 
with $\psi_c(x):=\psi(x+c)$ and $\lambda_c := \lambda - i c$.
Consequently, the spectrum of~$H$ is empty,
$$
  \sigma(H) = \varnothing
  \,.
$$
This is a peculiar property, possible for non-self-adjoint operators only.
We deduce that~$H$ is not similar to a self-adjoint operator,
via a bounded and boundedly invertible similarity transform.

\paragraph{$\bullet$ Pseudospectrum.}
While~$H$ has no spectrum, the pseudospectrum of~$H$
is far from trivial. 
\emph{A priori}, we only know that the pseudospectrum
is symmetric with respect to the real axis,
\ie~\eqref{antilinear.pseudo} holds,
which follows from the $\mathcal{PT}$-symmetry of~$H$.
In order to apply Theorem~\ref{Thm.DSZ},
we have to convert~$H$ into a semiclassical operator. 
This can be achieved by introducing the unitary transform~$\mathcal{U}$
on $\sii(\Real)$ defined by 
\begin{equation}\label{unitary}
  (\mathcal{U}\psi)(x) := \tau^{1/2} \psi(\tau x)
  \,,
\end{equation}
where $\tau \in \Real$ is positive (and typically large in the sequel).
Then 
$$
  \mathcal{U} H \mathcal{U}^{-1} = \tau H_h
  \qquad \mbox{with} \qquad
  H_h := - h^{2} \frac{\dd^2}{\dd x^2} + ix
  \qquad \mbox{and} \qquad
  h := \tau^{-3/2}
  \,.
$$
For the symbol~$f = \xi^2 + ix$ associated with~$H_h$ we have $\{f, \bar{f}\} = -4i\xi.$
Hence, the interior of the semiclassical pseudospectrum is
$\Lambda = \{z\in\Com \ | \ \Re z > 0 \}$, using definition \eqref{semiclass.pseudospec}.

The same translation argument which shows that the spectrum is empty shows that
\[
	\|(H-z)^{-1}\| = \|(H-\Re z)^{-1}\|.
\]
Note that $1 \in \Lambda$. Applying the unitary relation and Theorem~\ref{Thm.DSZ}, there exists some $C>1$ where, for $h$ sufficiently small (that is, $\tau > C_1$ for some $C_1 > 0$ sufficiently large),
$$
  \|(H-\tau)^{-1}\| 
  = \tau^{-1} \|(H_h-1)^{-1}\|
  > h^{2/3} C^{1/h}
  \,.
$$
We then have that $\tau \in \sigma_\eps(H)$ whenever $\tau^{-1}C^{\tau^{3/2}} > \eps^{-1}$.  We may simplify the inequality by taking logarithms: 
it reads
\[
	\tau^{3/2} - \frac{\log \tau}{\log C} > \frac{1}{\log C}\log\frac{1}{\eps}.
\]
Since $\log \tau$ is negligible compared with $\tau^{3/2}$ for $\tau > 0$ large, 
a sufficient condition to guarantee that $\tau \in \sigma_\eps(H)$ is given by
$$
\tau > C_2 \left(\log \frac{1}{\eps}\right)^{2/3}
$$
with some $C_2 >0$; this gives the correct order of growth as $\eps \to 0$.

Since the resolvent norm only depends on the real part of the spectral parameter $z$, we arrive at the conclusion that there exist $C_1, C_2 > 0$ such that, for all $\eps > 0$,
$$
  \sigma_\eps(H) \supseteq \left\{z \ \left| 
  \ \Re z \geq C_1 \ \& \ \Re z \geq C_2 \left(\log \frac{1}{\eps}\right)^{2/3} \right.
  \right\}.
$$
In particular, for any~$\eps$ there are complex points with positive real part
and magnitude only logarithmically large in $1/\eps$ that lie in the pseudospectrum $\sigma_\eps(H)$.

A quite precise study of the resolvent norm of $H$ as $\Re z \to \infty$ can be found in \cite[Cor.~1.4]{BordeauxMontrieux-2013}.

\paragraph{$\bullet$ Time evolution.}
Since~$H$ is $m$-accretive,
it is a generator of a one-parameter
contraction semigroup, $e^{-t H}$, on $\sii(\Real)$.
Here~$t$ can be interpreted as time, viewing $\psi(t,x) = e^{-t H}\psi(0,x)$ as 
arising as a solution of the parabolic equation $\partial_t \psi+H\psi=0.$
Using the Fourier transform,
it is possible to show (\cf~\cite[Ex.~9.1.7]{Davies_2007}) that
$$
  \|e^{-t H}\| = e^{-t^3/12}
  \,.
$$ 
Note that the time decay rate is not determined by 
the (nonexistent) spectrum of~$H$. 
In fact, the superexponential decay rate implies
by itself that the spectrum of~$H$ must be empty. 
We refer to \cite{Almog-Helffer-Pan_2013} for pseudospectral estimates 
on the decay of the semigroup in analogous higher-dimensional models.

\subsection{The imaginary cubic oscillator}
The non-self-adjoint (but $\mathcal{T}$-self-adjoint) operator 
\begin{equation}\label{cubic}
  H := - \frac{\dd^2}{\dd x^2} + ix^3
  \qquad \mbox{on} \qquad
  \sii(\Real)
\end{equation}
is considered the \emph{fons et origo}
of $\mathcal{PT}$-symmetric quantum mechanics
\cite{Bender-Boettcher_1998,BBJ},
but it was also considered previously in the context 
of statistical physics and quantum field theory~\cite{Fisher_1978}.
The existence of a metric operator 
and other spectral and pseudospectral properties of~$H$
have been recently analysed in~\cite{SK}.
Let us recall the basic results here,
referring to the last article for more details and references. 

\paragraph{$\bullet$ Spectrum.}
The operator~$H$ is again $m$-accretive 
when considered on its maximal domain 
(\ie~\eqref{Airy.dom} with~$x$ replaced by~$x^3$)
and its resolvent is compact.
Contrary to the imaginary Airy operator,
the spectrum is not empty; it is composed of
an infinite sequence of discrete real eigenvalues, 
\cf~\cite{Shin,DDT,Giordanelli-2014-of}.
As a new result, it is proved in~\cite{SK} that
the eigenfunctions form a complete set in~$\sii(\Real)$.

\paragraph{$\bullet$ Pseudospectrum.}
Employing the unitary transform~\eqref{unitary},
we introduce a semiclassical analogue of~$H$,
$$
  \mathcal{U} H \mathcal{U}^{-1} = \tau^3 H_h
  \qquad \mbox{with} \qquad
  H_h := - h^{2} \frac{\dd^2}{\dd x^2} + ix^3
  \qquad \mbox{and} \qquad
  h := \tau^{-5/2}
  \,.
$$
For the symbol~$f$ associated with~$H_h$ 
we now have $\{f,\bar{f}\}=-12i\xi x^2$, 
so that the interior of the semiclassical pseudospectrum is
$\Lambda = \{z\in\Com \ | \ \Re z > 0 \ \& \ \Im z \not=0 \}$.

The translation argument used for the imaginary Airy operator is unavailable for the imaginary cubic oscillator, but we nonetheless have 
by Theorem~\ref{Thm.DSZ}
for any $z \in \Lambda$, there exists $C > 0$ sufficiently large and $h_0$ sufficiently small that, for all $0 < h \leq h_0$,
\begin{equation}\label{res.ix3}
  \|(H-\tau^3 z)^{-1}\| 
  = \tau^{-3} \|(H_h-z)^{-1}\|
  > h^{6/5} C^{1/h}
  \,.
\end{equation}
One may check that the exponential growth in Theorem~\ref{Thm.DSZ} may be made uniform on compact subsets $K \subset \Lambda$, as explained at the end of Appendix \ref{App.proofs}.  What is more, one may extend the reasoning to include real $z > 0$, despite the fact that formally $z \notin \Lambda$: one only needs to verify \eqref{e.unif.2} by hand, which is straightforward. 

Since we are using a scaling argument, we may reduce to fixing $\delta > 0$ and letting
\[
	K = \{z \in \Com \ | \ |z| = 1 \ \& \ |\arg z| < \pi/2-\delta\}.
\]
We therefore have that, 
for some positive constant~$C$ depending on~$\delta$, 
the pseudospectrum $\sigma_\eps(H)$ contains $\tau^3 K$ so long as $\tau > 0$ is large enough to verify $\tau^{-3}C^{\tau^{5/2}} > \eps^{-1}$.
We may then identify $\tau^3$ with the absolute value of the spectral parameter, and so $h = \tau^{-5/2} = |z|^{-5/6}$. Taking logarithms and discarding the negligible term involving logarithms, as before, allows us to conclude that, for any $\delta > 0$, there exist $C_1, C_2 > 0$ such that, for all $\eps > 0$,
\begin{equation}\label{CubicPseudospec}
	\sigma_\eps(H) \supseteq \left\{z \in \Com \ \left| \ |z| \geq C_1 \ \& \  |\arg z| < \left(\frac{\pi}{2}-\delta\right) \ \& \ |z| \geq C_2 \left(\log \frac{1}{\eps}\right)^{6/5} \right. \right\}.
\end{equation}
Again, for any~$\eps$ there are complex points with positive real part, 
non-zero imaginary part, and large magnitude 
that lie in the pseudospectrum $\sigma_\eps(H)$.
A numerical computation of some pseudospectral lines of~$H$ 
is presented in Figure~\ref{Fig.cubic}. 
The asymptotic behaviour of the pseudospectral lines is
studied in \cite[Prop.~4.1]{BordeauxMontrieux-2013}.
\begin{figure}[h!]
\begin{center}
\includegraphics[width=0.99\textwidth]{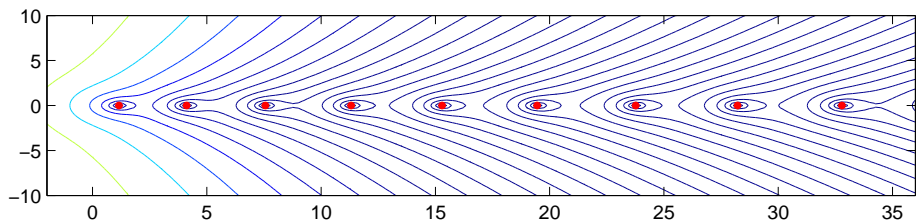}
\end{center}
\caption{Spectrum (red dots) and pseudospectra 
(enclosed by the blue contour lines)
of the imaginary cubic oscillator.}
\label{Fig.cubic}
\end{figure}
The result is surprising because it implies the existence
of pseudoeigenvalues very far from the spectrum of~$H$.
In view of~\eqref{RS-thm}, it follows that a very small 
perturbation~$V$ added to~$H$ can create (genuine) eigenvalues
very far from the spectrum of the unperturbed operator~$H$. 
In this way, the spectrum is highly unstable.

As a consequence of the existence of the highly non-trivial pseudospectrum,
we also get that~$H$ is not quasi-Hermitian with a bounded
and boundedly invertible metric (Proposition~\ref{Prop.test}),
it is not similar to a self-adjoint operator via bounded
and boundedly invertible transformations (Proposition~\ref{Prop.equivalence}),
and the eigenfunctions of~$H$ do not form a Riesz basis
(Proposition~\ref{Prop.Riesz}). 
It has been shown recently in~\cite{Henry-2013b} 
that the norms of the spectral projections grow as
\begin{equation}
\lim_{k \to \infty} \frac{\log \|P_k\|}{k} = \frac{\pi}{\sqrt{3}},
\end{equation}
and therefore the eigenfunctions cannot form a basis, 
\cf~Proposition~\ref{Prop.basis}. 
Alternatively, we can derive the latter 
using Proposition~\ref{Prop.basis.D} and~\eqref{res.ix3}.

\subsection{An advection-diffusion operator}
The examples which follow are similar to self-adjoint operators
via \emph{unbounded} transformations. In order to emphasize the danger
of formal manipulations when the transformations are allowed to be unbounded,
in this subsection we present a very simple non-self-adjoint operator
for which heuristic approaches would lead to a number of striking paradoxes.
The example is borrowed from \cite{Reddy-Trefethen_1994,Davies_2002},
although it is natural to expect that it has appeared in many other works.
We are indebted to E.~B.~Davies and M.~Marletta~\cite{DM_2013}
for telling us about this example and for proposing
the possibility of the non-invariance of point spectra
discussed below.

Consider the differential operator
\begin{equation}\label{advection}
  H := - \frac{\dd^2}{\dd x^2} + \frac{\dd}{\dd x}
  \qquad \mbox{on} \qquad
  \sii(\Real)
  \,.
\end{equation}
The diffusion term $-\dd^2/\dd x^2$ corresponds to the familiar 
free Hamiltonian in quantum mechanics, 
which is self-adjoint when defined on
$
W^{2,2}(\Real) 
$.
The advection term represents a relatively bounded perturbation
with relative bound equal to zero, 
so that $\Dom(H)=W^{2,2}(\Real)$,
\cf~\cite[Sec.~IV.1.1]{Kato}.
Employing the first line of~\eqref{for.advection}, 
we find 
$$
  |\Im (\psi,H\psi)| 
  \leq \|\psi\| \|\psi'\|
  = \|\psi\| \sqrt{\Re (\psi,H\psi)}
$$
for every $\psi \in \Dom(H)$.
It follows that the numerical range $\Num(H)$ 
is contained in the parabolic domain 
$  
  \Sigma := \{z \in \Com \ | \ 
  \Re z \geq 0 \ \& \ \ |\Im z|^2 \leq \Re z \}
$.
In particular, $H$~is not only $m$-accretive but even $m$-sectorial,
meaning that, in addition, its numerical range is a subset of a sector
in the complex plane.
By conjugating with the Fourier transform,
it is easy to check that the spectrum of~$H$
coincides with the parabola
$$
  \sigma(H) = \partial\Sigma 
  = \{z \in \Com \ | \ 
  \Re z \geq 0 \ \& \ |\Im z|^2 = \Re z \}
  \,,
$$
see Figure~\ref{Fig.advection}.

\begin{figure}[h!]
\begin{center}
\includegraphics[width=0.99\textwidth]{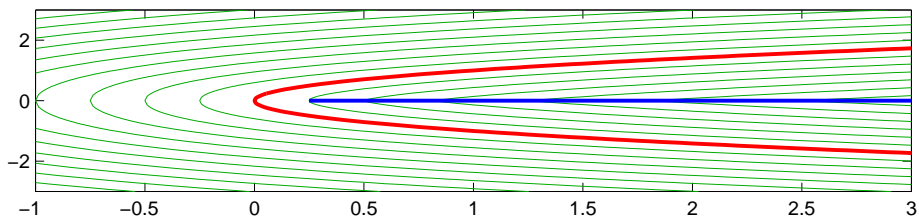}
\end{center}
\caption{Spectrum of the advection-diffusion operator (red parabola) 
compared with the spectrum of the shifted free Hamiltonian (blue line)
to which the former is formally similar
and pseudospectral contours (green curves parallel
to the parabola).}\label{Fig.advection}
\end{figure}

\paragraph{$\bullet$ Non-invariance of the continuous spectrum.}
Completing the square, we may write
$$
  H = - \left(\frac{\dd}{\dd x}-\frac{1}{2}\right)^2 + \frac{1}{4}
  \,.
$$
This suggests that the similarity transformation $\Omega := e^{-x/2}$ 
\emph{formally} maps~$H$ to a shifted free Hamiltonian
\begin{equation}\label{advection.similar}
  \Hsa := - \frac{\dd^2}{\dd x^2} + \frac{1}{4}
  \,,
\end{equation}
which is self-adjoint on 
$\Dom(\Hsa) := W^{2,2}(\Real)$.
The word ``formally'' is absolutely essential here,
since neither~$\Omega$ nor $\Omega^{-1}$ is bounded
and~\eqref{similar} cannot hold as an operator identity.
Nevertheless, one can check that~\eqref{similar} holds
on smooth functions with compact support, 
which form a dense subset of both~$\Dom(H)$ and~$\Dom(\Hsa)$. 
Now we arrive at a surprising paradox because
$$
  \sigma(\Hsa) = [\mbox{$\frac{1}{4}$},\infty)
$$
substantially differs from the complex parabolic spectrum of~$H$, 
see Figure~\ref{Fig.advection}.
This is caused by the fact that the continuous spectrum
is in general not preserved by unbounded similarity transformations. 

\paragraph{$\bullet$ Non-invariance of the point spectrum.}
The situation is in fact even worse, since it may also happen
that even the eigenvalues are not preserved 
by the similarity transformation~$\Omega$.
Let us perturb the self-adjoint Hamiltonian~$H_0$ 
by a smooth non-trivial potential $V \leq 0$ 
which has a compact support in~$\Real$. 
Then it is well known (see, \eg, \cite[Cor.~4.5.2]{Schechter}) 
that the self-adjoint operator $\Hsa +V$ 
possesses at least one eigenvalue $\lambda < 1/4$.
The corresponding eigenfunction~$\psi$ has asymptotics
$\exp({-\sqrt{1/4-\lambda}\,|x|})$ as $x \to \pm\infty$.
It is also known that~$\lambda$ is positive 
provided that the supremum norm of~$V$ is sufficiently small, 
\cf~\cite{Si}. 
However, the corresponding ``eigenfunction'' $\Omega^{-1}\psi$ of~$H+V$
is not admissible because it is not square-integrable on~$\Real$.

\paragraph{$\bullet$ Pseudospectrum.}
Even if~$H$ is not self-adjoint
(nor $\mathcal{T}$-self-adjoint or $\mathcal{PT}$-symmetric), 
it is normal (and in fact real).
Hence the pseudospectra are trivial;
see Figure~\ref{Fig.advection}.
The situation would be very different if 
we considered~\eqref{advection} on a finite interval $(0,L)$,
subject to Dirichlet boundary conditions.
Let us denote this operator by~$H^{(L)}$.
Then the similarity transformation $\Omega=e^{-x/2}$
is bounded and boundedly invertible and~\eqref{similar} 
with the self-adjoint operator~$\Hsa^{(L)}$ 
that acts as~\eqref{advection.similar} on $(0,L)$,
subject to Dirichlet boundary conditions, is well defined.
We indeed have 
$$
  \sigma(H^{(L)}) = \sigma(\Hsa^{(L)}) 
  = \bigg\{ 
  \left(\frac{\pi k}{L}\right)^2+\frac{1}{4} \ \bigg| \ k=1,2,\dots
  \bigg\}
  .
$$
However, the pseudospectra of~$H^{(L)}$ substantially differ 
from those of~$\Hsa^{(L)}$; the former approach the parabola 
$\partial\Sigma$ in the limit as $L \to \infty$,
thus reflecting better the wild spectral instability of the limit;
see Figure~\ref{Fig.advection}. 
We refer to~\cite{Reddy-Trefethen_1994} for more details.

\subsection{The rotated harmonic oscillator}
\label{subsec.rot.ho}
The quantum Hamiltonian of the harmonic oscillator 
\begin{equation}\label{HO}
  \Hsa := - \frac{\dd^2}{\dd x^2} + x^2
  \qquad \mbox{on} \qquad
  \sii(\Real)
\end{equation}
is self-adjoint on its maximal domain.
The operator~$\Hsa$ has compact resolvent 
and its eigenvalues are well-known:
\begin{equation}\label{HO.spec}
  \sigma(\Hsa) = \{2 k + 1 \ | \ k=0,1,\dots\}
  \,.
\end{equation}
Of course, the pseudospectrum of~$\Hsa$ is trivial;
see~Figure~\ref{Fig.HO-like}.

\paragraph{$\bullet$ Creation and annihilation operators.}
Recall also the factorization $\Hsa = a^* a+1$,
where~$a^*$ and~$a$ are the creation and annihilation
operators defined by
\begin{equation}\label{creation}
  a^* := - \frac{\dd}{\dd x} + x
  \,, \qquad
  a := \frac{\dd}{\dd x} + x
  \,,
\end{equation}
with $\Dom(a)=\Dom(a^*):= 
W^{1,2}(\Real)
\cap \sii(\Real,x^2\,\dd x)$.
As the notation suggests, $a^*$~and~$a$ are mutually adjoint.  
The operator domain $\Dom(a)$ in fact coincides 
with the form domain of~$\Hsa$, 
for which $C_0^\infty(\Real)$ is a core.
We incidentally remark that $a^*$~and~$a$ represent  
interesting examples of highly non-self-adjoint operators
for which the resolvent operator is not defined 
at any point of the complex plane:
$$
  \sigma(a^*) = \sigma(a) = \Com
  \,.
$$
Indeed, $\psi(x):=\exp{(\lambda x -x^2/2)}$ is an eigenfunction of~$a$
for every $\lambda \in \Com$. 
On the other hand, the point spectrum of~$a^*$ is empty,
but every complex point belongs to the residual spectrum of~$a^*$.
The latter follows by the general fact that the orthogonal complement 
of the range of a densely defined closed operator in a Hilbert space 
is equal to the kernel of its adjoint.

\paragraph{$\bullet$ Complex dilation.}
The rotated harmonic oscillator~$H$ is \emph{formally} obtained 
by the similarity transformation~\eqref{similar},
where~$\Omega^{-1}$ is the complex dilation operator defined by
$
  (\Omega^{-1}\psi)(x) := e^{i\theta/4} \psi(e^{i\theta/2} x)
$. 
If~$\theta$ were purely imaginary, $\Omega^{-1}$~would be unitary,
but we assume $\theta \in \Real$ 
when~$\Omega^{-1}$ is in fact unbounded.
The formal computation yields
\begin{equation}\label{HO.rotated}
  H = - e^{-i\theta} \frac{\dd^2}{\dd x^2} + e^{i\theta} x^2
  = e^{-i\theta} \left(
  - \frac{\dd^2}{\dd x^2} + e^{2i\theta} x^2
  \right)
  \qquad \mbox{on} \qquad
  \sii(\Real)
  \,,
\end{equation}
which we take as a definition and restrict to $|\theta| < \pi/2$. 
The operator~$H$ is sometimes referred to as Davies' oscillator
due to his important investigation~\cite{Davies_1999a}
(see also \cite[Sec.~14.5]{Davies_2007} for a summary
and other references).
The presence of the prefactor $e^{-i\theta}$ in our definition
is inessential, but it is useful for symmetry reasons.
In particular, we have
$$
  |\Im (\psi, H\psi)| 
  \leq |\tan\theta| \ \Re (\psi,H \psi)
  = |\sin\theta| \ (\psi,\Hsa\psi)
$$
for every $\psi \in C_0^\infty(\Real)$.
That is, under our restriction on~$\theta$, the operator
$H$~can be understood as obtained as a relatively 
form bounded perturbation of~$\Hsa$,
keeping the same form domain $\Dom(a)$.
The numerical range $\Num(H)$ lies in the symmetric sector
$
  \{z\in\Com \ | \ |\Im z| \leq |\tan\theta| \Re z \ \& \ \Re z \geq 0 \}
$.
Unless $\theta=0$, $H$~is neither self-adjoint nor $\mathcal{PT}$-symmetric,
but it is always $\mathcal{T}$-self-adjoint.

\paragraph{$\bullet$ Spectrum and pseudospectrum.}
The resolvent of~$H$ is again compact 
and the spectrum coincides with that of~$\Hsa$:  
\begin{equation}\label{sp.invariance}
  \sigma(H) = \sigma(\Hsa)
  \,.
\end{equation}
In particular, the spectrum is real. 
However, the pseudospectrum and basis properties of the eigenfunctions
are very different from the self-adjoint situation.
In the same way as we applied Theorem~\ref{Thm.DSZ}
to the imaginary Airy operator or cubic oscillator,
we find 
\[
	\Lambda = \{z\in\Com\backslash\{0\} \ | \ |\arg z| < \theta \}.
\]
When $\theta = 0$, $\Lambda$ is empty and~$\Hsa$ is self-adjoint, 
so that we know that its pseudospectrum is trivial.
If $\theta\not=0$, however, we have exponential growth for the resolvent indicated by Theorem~\ref{Thm.DSZ}, which may be made uniform for a compact subset contained in the interior of the semiclassical pseudospectrum. What is more, from \cite[Thm.~1.1, Rem.~1.3]{Hitrik-Sjostrand-Viola_2013}, we have upper bounds for the resolvent norm of essentially the same exponential type, though the gap in the constant leaves much to be understood about the precise behaviour.  

A scaling argument similar to that used for \eqref{CubicPseudospec} gives us an idea of size of the $\eps$-pseudospectrum, 
which includes more or less those $z \in \Num(H)$ for which $|z| > C\log(1/\eps)$, 
which grows very slowly as $\eps \to 0$.

Specifically, for any $\delta > 0$, we have from Theorem~\ref{Thm.DSZ} that there exist $C_1, C_2 > 0$ such that, for all $\eps > 0$,
\[
	\sigma_\eps(H) \supseteq \left\{z \ \left| \ |z| > C_1 \ \& \ |\arg z| \leq \theta - \delta \ \& \ |z| > C_2\log\frac{1}{\eps}\right. \right\}.
\]
The numerical range inequality \eqref{robust} gives that
\[
	\sigma_\eps(H) \subseteq 
	\left\{z \ \left| \ \dist(z,\Num(H)) < \eps\right.\right\}.
\]
What is more, rescaling the upper bound \cite[Eq.~(1-8)]{Hitrik-Sjostrand-Viola_2013} gives that there exists $C_3 > 0$ for which
\[
	\left\|(H- z)^{-1}\right\| \leq \frac{C_3 e^{C_3|z|}}{\dist(z,\sigma(H))}.
\]
We therefore see that $z$ cannot be in the pseudospectrum unless 
it is logarithmically large in $1/\eps$ or close to the spectrum: 
\[
	\sigma_\eps(H) 
        \subseteq \left\{z \, \left| \ |z| > 
        \frac{1}{C_3}\left(\log\frac{1}{\eps}-\log C_3\right)\right.\right\} 
	\cup \left\{z \, \left| \ \dist(z, \sigma(H)) 
        < C_3\eps e^{C_3|z|}\right.\right\}.
\]
In Figure \ref{Fig.Containment}, 
we have a diagram illustrating the sorts of regions 
containing and contained by the pseudospectrum. 
We emphasize that the constants involved were chosen 
by hand and do not reflect a precise application of the relevant theorems.
\begin{figure}
\centering
\includegraphics[width=0.8\textwidth]{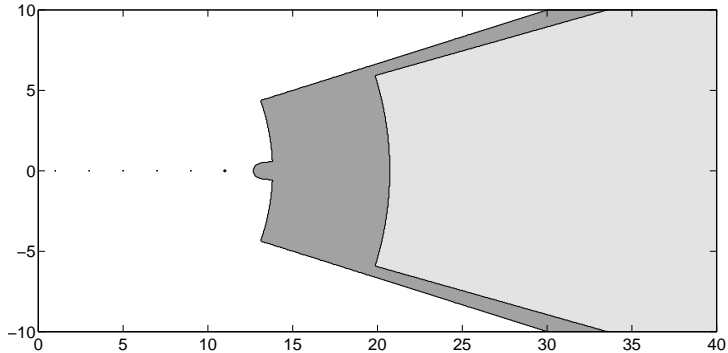}
\caption{Diagram of region (dark gray) which contains $\sigma_\eps(H)$ and region with pseudomodes (light gray) which is definitely contained in $\sigma_\eps(H)$.}
\label{Fig.Containment}
\end{figure}
As in the previous examples, 
for any~$\eps > 0$ there are complex points with positive real part, 
non-zero imaginary part, and large magnitude 
that lie in the pseudospectrum $\sigma_\eps(H)$;
see Figure~\ref{Fig.HO-like}.
The asymptotic behaviour of the pseudospectrum 
has been also studied in 
\cite{Pravda-Starov_2006,BordeauxMontrieux-2013}.

\paragraph{$\bullet$ Wild basis properties.}
$H_0$~is self-adjoint and we know \emph{a priori}
that its eigenfunctions (after normalization)
form a complete orthonormal family in~$\sii(\Real)$. 
If~$\theta\not=0$, it is still true that
the eigenfunctions of~$H$ and~$H^*$
(after a suitable normalization) form a biorthonormal sequence. 
However, the eigenfunctions do not form
a Riesz basis or even a basis, 
as can be deduced from the non-trivial pseudospectrum 
(\cf~Propositions~\ref{Prop.Riesz} and~\ref{Prop.basis.D}).

Furthermore, it was shown in~\cite{Davies-Kuijlaars_2004} 
that the norms of the spectral projections of~$H$ 
associated with eigenvalues $2k+1$ grow exponentially
as $k \to \infty$, and the exact exponential rate of growth was identified.
This exponential growth rate was sharpened in~\cite[Thm.~1.2, Rem.~1.3]{Henry} 
to include an asymptotic expansion for the remainder and a generalization to operators including
$-\dd^2/\dd x^2 + e^{2 i \theta} |x|^{2 m}$ for $m \in \Nat$ and natural restrictions on~$\theta$.  

In the case of the rotated harmonic oscillator, one has from \cite[Cor.~1.7, Ex.~3.6]{Viola_2013} an exact integral formula leading to a similar asymptotic expansion and a simplification of the formula for the exponential growth rate to
\begin{equation}\label{rot.osc.rate}
\lim_{k \to +\infty} \frac{\log \|P_k\|}{k} = \frac{1}{2}\log\left(\frac{1+|\sin \theta|}{1-|\sin\theta|}\right).
\end{equation}
The new and very short proof of this formula using special functions is given in Appendix \ref{app.Pk}.

Summing up,
despite the reality of the spectrum of~$H$
and the existence of a formal similarity transformation
given by a complex dilation, 
we see that~$H$ exhibits very different properties 
from those enjoyed by self-adjoint operators.

\subsection{The shifted harmonic oscillator}
\label{subsec.shift.ho}
The shift operator~$T_t$ is, for $t\in\Real$,
well defined as a bounded operator on~$\sii(\Real)$
by the formula $(T_t\psi)(x):=\psi(x+t)$.  
It is in fact a unitary group 
$T_t = e^{t \, \dd/\dd x}$
with the self-adjoint generator~$-i \, \dd/\dd x$ 
defined on $W^{1,2}(\Real)$ 
being the familiar momentum operator 
in quantum mechanics.
%
%
Let~$\Hsa$ be again the self-adjoint harmonic oscillator~\eqref{HO}
and perform the \emph{formal} conjugation~\eqref{similar}
with the \emph{unbounded} operator $\Omega := T_{-i}$.
Then we arrive at the non-self-adjoint operator
\begin{equation}\label{HO.shifted}
  H = - \frac{\dd^2}{\dd x^2} + (x+i)^2
  \qquad \mbox{on} \qquad
  \sii(\Real)
  \,,
\end{equation}
which we again take as a definition.
The extra term $2i x - 1$ clearly represents 
a relatively bounded perturbation of~$\Hsa$
with the relative bound equal to zero,
so that $\Dom(H)=\Dom(\Hsa)$.
It is also relatively form bounded 
with relative bound equal to zero,
so that~$H$ is $m$-sectorial on $\Dom(a)$.
The operator~$H$ is $\mathcal{PT}$-symmetric
and $\mathcal{T}$-self-adjoint.

\paragraph{$\bullet$ Spectrum and pseudospectrum.}
The above remarks on the smallness of the perturbations
imply that the resolvent of~$H$ is compact. 
Notice that conjugation with the Fourier transform casts~$H$ 
into a unitarily equivalent operator 
$$
  \hat{H}   := \left(-i\frac{\dd}{\dd x}-i\right)^2+x^2
  \,,
$$
which is related to~\eqref{HO} via the unbounded
similarity transform $\hat{\Omega}\psi(x):=e^x\psi(x)$.
Using that the eigenfunctions of the harmonic oscillator~\eqref{HO}
are known to decay superexponentially, it can be showed that
$$
  \sigma(H) = \sigma(\Hsa)
  \,,
$$ 
analogously to~\eqref{sp.invariance}.
 
Numerical computations show that the pseudospectrum of~$H$ is non-trivial;
see Figure~\ref{Fig.HO-like}.
Because of the different scaling properties of~$x^2$ and~$x$, 
Theorem~\ref{Thm.DSZ} is not directly applicable. Nonetheless,
we show, by adapting the construction of pseudomodes, that there
are indeed large complex pseudoeigenvalues in parabolic regions of the complex plane.

\begin{Theorem}\label{Thm.Shifted}
 	Let $H$ be the operator from \eqref{HO.shifted}.
	Fix $\eps > 0$.  Then there exists $C > 0$ sufficiently large 
 such that for all $z \in \Com$ for which $\Re z > C$ and
	\[
	 	|\Im z| \leq 2(1-\eps)\sqrt{\Re z},
	\]
	we have the resolvent lower bound
	\[
	 	\|(H-z)^{-1}\| \geq \frac{1}{C}e^{\sqrt{\Re z}/C}.
	\]
\end{Theorem}

We postpone the proof to Appendix \ref{App.proofs}.

\begin{figure}[h!]
\begin{center}
\includegraphics[width=0.99\textwidth]{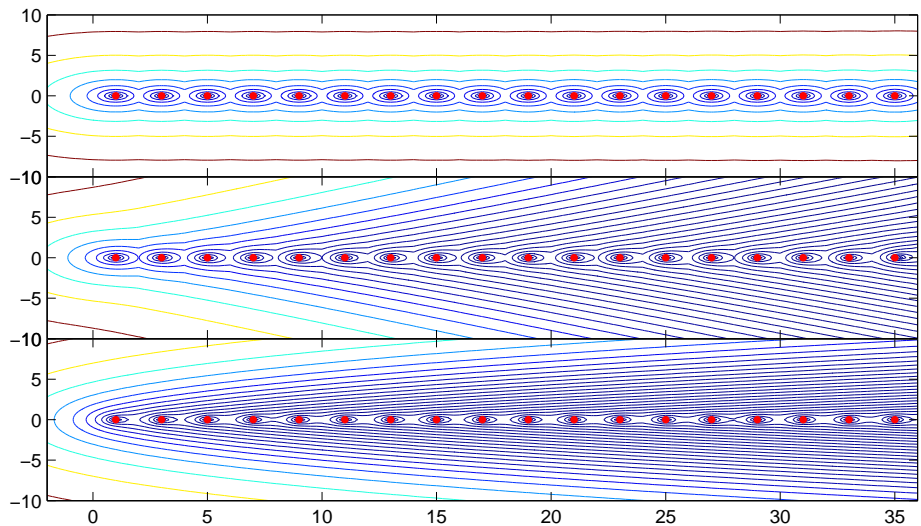}
\end{center}
\caption{From top to bottom: Pseudospectra of 
the self-adjoint~\eqref{HO}, 
rotated~\eqref{HO.rotated} with $\theta=\pi/4$ 
and shifted~\eqref{HO.shifted} 
harmonic oscillators which are formally conjugate to each other.
Although the spectra (red dots) coincide in all these examples,
the pseudospectra exhibit striking differences.}\label{Fig.HO-like}
\end{figure}

\paragraph{$\bullet$ Wild basis properties.}
The pseudospectral criteria of Propositions~\ref{Prop.test},~\ref{Prop.Riesz} 
and~\ref{Prop.basis.D} show that the eigenfunctions of~$H$ 
cannot form a Riesz basis or even a basis. 
Moreover, explicit formulas for the eigenfunctions can be used to prove their completeness in $L^2(\mathbb{R})$ and to find the rate of the norms of spectral projections, namely 
\begin{equation}\label{shift.osc.rate}
\lim_{k \to +\infty} \frac{\log \|P_k\|}{\sqrt{k}} = 2^{3/2}, \end{equation}
see \cite[Sec.~2]{Mityagin-2013} for the detailed proof.

\subsection{The decaying and singular potential perturbations of harmonic oscillator}

Examples in previous sections show that the perturbations 
of harmonic oscillator may have very non-trivial pseudospectra 
and wild basis properties, although the spectrum is preserved. 
In order not to leave an impression that this is typical 
for any perturbation of the harmonic oscillator, 
we mention that the system of all generalized eigenfunctions of perturbations 
$H = \Hsa + V$ of $\Hsa$ in~\eqref{HO} contains a Riesz basis 
when the perturbation~$V$ is multiplication 
by a function satisfying, for instance, 
$V \in L^p(\Real), 1 \leq p < \infty,$ or $V$ is a finite linear combination of 
$\delta$-potentials with complex couplings; \cf~\cite{Adduci-2012-10, 
Adduci-Mityagin_2012, Mityagin-2013a} for details and additional examples.
It is also showed in these works that the eigenvalues 
of the perturbed operator~$H$, 
excluding possibly a finite number, remain simple. Therefore, if the 
perturbation satisfies some symmetries, \eg~it is $\PT$-symmetric, we can conclude 
that the eigenvalues of~$H$ are real, again up to a finite number. 
Moreover, if we insert a sufficiently small coupling constant $\varepsilon$ in front 
of $V$, we obtain that all eigenvalues are simple and real. When the coupling 
constant is increased, low lying eigenvalues may collide, create a Jordan block, and 
then become complex. This behaviour is illustrated in Figure \ref{Fig.HO-pert.sp} for an example of~$H$ where~$V$ is $\PT$-symmetric and has the form \eqref{HO.pert.ex}.
\begin{figure}[htb!]
\begin{center}
\includegraphics[width=0.6\textwidth]{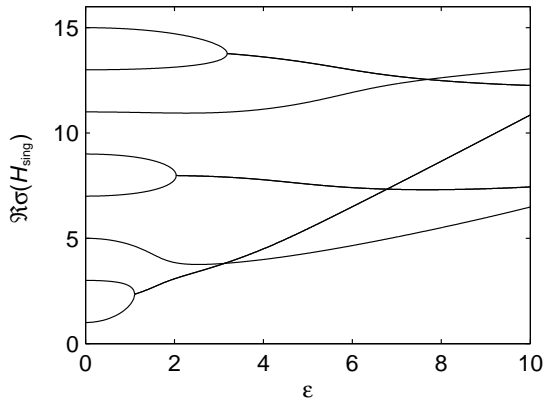} 
\end{center}
\caption{Real parts of eigenvalues of~$H$ with  \eqref{HO.pert.ex} 
as a function of~$\varepsilon$.
}
\label{Fig.HO-pert.sp}
\end{figure}
\paragraph{$\bullet$ Pseudospectrum.}
Let us recall that the facts that the eigensystem contains a Riesz basis and the spectral properties described above imply that~$H$ is similar to a diagonal operator up to a possible finite number of finite dimensional Jordan blocks corresponding to the low lying eigenvalues. This in turn means that the pseudospectrum of~$H$ 
is trivial if there are no Jordan blocks. 
If $\lambda_0$ is a degenerate eigenvalue for which the algebraic multiplicity is strictly larger than the geometric one, the Jordan block result in a more singular behaviour of the resolvent, \ie~
\[
	\|(H-z)^{-1}\| \sim |\lambda_0-z|^{-n}, \qquad n>1,
\]
around $\lambda_0$. This is visible in the plot of the pseudospectrum, since the peak around the degenerate eigenvalue $\lambda_0$ is wider, \ie~the level lines are further from $\lambda_0$, and we can call the pseudospectrum ``almost trivial''.
We illustrate such a behaviour with the example
\begin{equation}\label{HO.pert.ex}
V(x):= 
i \varepsilon \left( \frac{1}{\sqrt{|x+1|}} - \frac{1}{\sqrt{|x-1|}} \right),
\end{equation}
see Figure \ref{Fig.HO-pert}.
\begin{figure}[htb!]
\begin{center}
\includegraphics[width=0.3\textwidth]{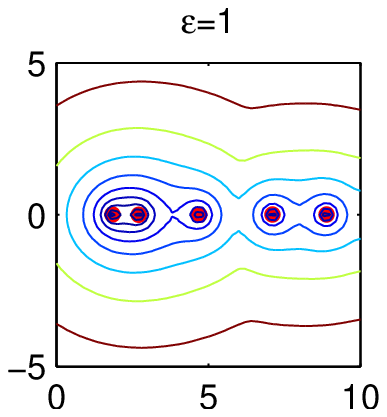} 
\quad 
\includegraphics[width=0.3\textwidth]{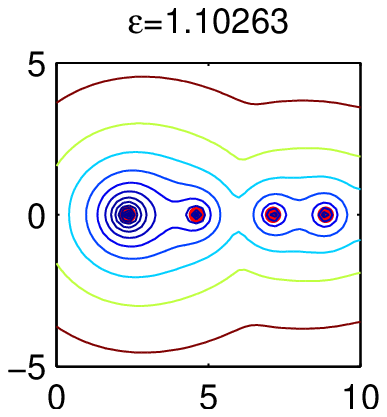} 
\quad
\includegraphics[width=0.3\textwidth]{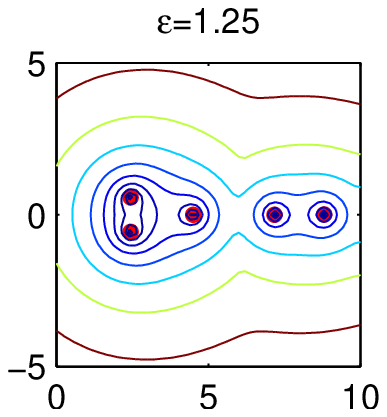}
\end{center}
\caption{Pseudospectra around low lying eigenvalues of \eqref{HO.pert.ex} for three increasing values of $\varepsilon$. From left to right: The simple eigenvalues (red dots) collide and create a Jordan block, then become again simple, but complex. The Jordan block structure is indicated by the wider peak around the degenerate eigenvalue.
}
\label{Fig.HO-pert}
\end{figure}

Similar features appear in for instance a model on a finite interval with $\PT$-symmetric Robin boundary conditions, \cf~\cite{KBZ,K4,KS,KSZ}, or for a $\PT$-symmetric square-well, \cf~\cite{Znojil_2001, Siegl-2011-50}. 

\subsection{The gauged oscillator (Swanson's model)}\label{subsec.Swanson}
We also study the operator
\begin{equation}\label{Swanson}
  H:=\omega \, a^*a + \alpha \, a^2 + \beta \, (a^{*})^2 + \omega
  \quad \mbox{in} \quad
  \sii(\Real)
  \,,
\end{equation}
introduced by Ahmed in \cite{Ahmed-2002-294} 
and later studied by Swanson in~\cite{Swanson_2004}. 
Here the creation and annihilation operators 
are defined in~\eqref{creation}
(the reader is warned that
a different convention is used in~\cite{Swanson_2004})
and $\omega$, $\alpha$, $\beta$ are real parameters.
It is assumed in \cite{Swanson_2004} only that
\begin{equation}\label{Swanson.cond1}
  \omega^2 - 4 \alpha\beta \geq 0
  \,,
\end{equation}
but we shall see that this condition is by far insufficient
to make the results rigorous.

\paragraph{$\bullet$ Rigorous definition of the operator.}
First of all, we always assume $\omega \not=0$ in order to 
define~$H$ as a perturbation of the harmonic oscillator~$\Hsa$ 
defined in~\eqref{HO};
without loss of generality, let us take $\omega>0$.
Second, we need to impose a condition on the smallness of~$\alpha$ and~$\beta$
in order to ensure that the extra unbounded terms 
$r := \alpha \, a^2 + \beta \, (a^{*})^2$
added in~\eqref{Swanson} to $\omega\,a^*a = \omega \Hsa - \omega$
do not completely change the character of 
the operator~$\Hsa$; 
see Section~\ref{subsec.el.q.op} for a discussion of 
the relevant condition in a more general setting.

Expressing annihilation and creation operators 
in terms of $x$ and $\dd/\dd x$, 
we obtain an equivalent form
\begin{equation}
H = -(\omega + \alpha + \beta) \frac{\dd^2}{\dd x^2} + (\omega - \alpha - \beta) x^2 + (\alpha - \beta) (\frac{\dd}{\dd x}x +x\frac{\dd}{\dd x}) 
\quad \mbox{on} \quad
  \sii(\Real).
\end{equation}
We introduce $H$ as an $m$-sectorial operator with compact resolvent under the condition 
\begin{equation}\label{Swanson.cond2}
  \omega - |\alpha + \beta| > 0
  .
\end{equation}
Notice that this condition is stronger than~\eqref{Swanson.cond1} and it
ensures that the real part of~$H$
indeed resembles the usual harmonic oscillator, 
because the constants in front of 
$-\dd^2/\dd x^2$ and $x^2$ are both positive. 

We define a form 
\begin{equation}
\begin{aligned}
t[\psi] & :=  (\omega + \alpha + \beta) \|\psi'\|^2 + (\omega - \alpha - \beta) \|x \psi\|^2 
\\ & \qquad 
+ (\alpha - \beta) ( \langle \psi', x \psi \rangle  + \langle x \psi, \psi' \rangle ),
\\
\Dom(t) & :=\Dom(a) 
= \{\psi \in W^{1,2}(\mathbb{R}) \, | \, x \psi \in L^2(\mathbb{R})\}.
\end{aligned}
\end{equation}
Since the imaginary part of $t$ satisfies
\begin{equation}
\begin{aligned}
| \Im t[\psi] | &= 
| (\alpha - \beta) (\langle \psi', x \psi \rangle  + \langle x \psi, \psi' \rangle ) | 
\\
& \leq |\alpha-\beta| (\|\psi'\|^2 + \|x\psi\|^2 ) 
\\
& \leq \frac{|\alpha-\beta|}{\omega -|\alpha + \beta|}  ((\omega + \alpha + \beta) \|\psi'\|^2 + (\omega - \alpha - \beta) \|x \psi\|^2)
\\
& 
= \frac{|\alpha-\beta|}{\omega -|\alpha + \beta|} \Re t[\psi],
\end{aligned}
\end{equation}
the form $t$ is sectorial. Moreover, it is closed, since $\Re t$ is closed on the given domain. The operator $H$ is then defined via the first representation theorem, \cf~\cite[Thm.~VI.2.1]{Kato}, as the unique $m$-sectorial operator associated with $t$. The operator $H$ has a compact resolvent, which follows from \cite[Thm.~VI.3.3]{Kato} and the fact that $\Re H$ has a compact resolvent; notice that $\Re H$ is an operator associated with $\Re t$ and it can be verified that 
\[
\Re H = - (\omega + \alpha + \beta) \, \frac{\dd^2}{\dd x^2}
  + (\omega - \alpha - \beta) \, x^2,
\]
as expected.

\paragraph{$\bullet$ Spectrum and pseudospectrum.}
Since the resolvent of $H$ is compact, the spectrum of $H$ is discrete.
To find the eigenvalues of~$H$, we observe that~$H$
is \emph{formally} similar to a self-adjoint harmonic oscillator.
Indeed, substituting~\eqref{creation} to~\eqref{Swanson}
and completing the square, we find
\begin{equation}\label{Swanson.complete}
  H = - (\omega-\alpha-\beta) 
  \left(
  \frac{\dd}{\dd x} + \frac{\beta-\alpha}{\omega-\alpha-\beta}\ x
  \right)^2
  + \frac{\omega^2 - 4 \alpha\beta}{\omega-\alpha-\beta} \ x^2
  \,.
\end{equation}
From this formula it is clear that~$H$ is self-adjoint
if, and only if, $\alpha=\beta$.
For $\alpha\not=\beta$, $H$~is neither self-adjoint
nor $\mathcal{T}$-self-adjoint,
but it is $\mathcal{PT}$-symmetric and in fact real. 
The difference $\alpha-\beta$ acts as a sort
of imaginary magnetic field.
The magnetic field can be gauged out in one dimension;
employing the same gauge transform for 
our ``imaginary magnetic field,'' 
we \emph{formally} check that $\tilde{H}_{\rm sa} = \Omega H \Omega^{-1}$ with
$$
  \tilde{H}_{\rm sa} := - (\omega-\alpha-\beta) \,
  \frac{\dd^2}{\dd x^2} 
  + \frac{\omega^2 - 4 \alpha\beta}{\omega-\alpha-\beta} \ x^2
  \,, \qquad
  \Omega := \exp{\left(
  \frac{\beta-\alpha}{\omega-\alpha-\beta} \, \frac{x^2}{2}
  \right)}
$$
The word ``formal'' refers again to the fact 
that~$\Omega$ is unbounded.
Nevertheless, the similarity relation 
$\tilde{H}_{\rm sa} = \Omega H \Omega^{-1}$ 
is well defined on eigenfunctions of~$\tilde{H}_{\rm sa}$
and we can deduce that
$$
  \sigma(H) = \sigma(\tilde{H}_{\rm sa}) 
  = \left\{ \left.
  (2 k + 1)\sqrt{\omega^2 - 4 \alpha\beta} 
  \ \right| \ k=0,1,\dots\right\}
  \,,
$$ 
since eigenfunctions of $H$ are complete in $L^2(\mathbb{R})$. The latter can be verified by adapting the standard proof of completeness of Hermite functions; see \eg~\cite[Ex.~2.2.3]{BEH.Springer}.

Large pseudoeigenvalues
can be shown to exist by applying Theorem~\ref{Thm.DSZ}, as before. 
Employing the unitary transform~\eqref{unitary},
we introduce a semiclassical analogue of~$H$
via $\mathcal{U} H \mathcal{U}^{-1} = \sigma^2 H_h$, 
where~$H_h$ is~\eqref{Swanson.complete} 
with the small number~$h:= \sigma^{-2}$
put in front of the derivative.
For the symbol~$f$ associated with~$H_h$ we get
$$
  \{f,\bar{f}\} = 8 i (\alpha-\beta)
  \left[
  (\omega+\alpha+\beta) \, \xi^2 - (\omega-\alpha-\beta) \, x^2
  \right]
  \,.
$$
Consequently, $\Lambda$~has the same structure as that of the rotated harmonic oscillator:
a cone in the right complex half-plane with
two semiaxes removed.
Applying Theorem~\ref{Thm.DSZ}, we see that there is exponentially rapid resolvent growth in the interior of the cone.
Again, for any~$\eps > 0$ there are complex points with positive real part, 
non-zero imaginary part, and large magnitude 
that lie in the pseudospectrum $\sigma_\eps(H)$.

Although the straightforward application of Theorem~\ref{Thm.DSZ} already shows non-trivial character of the pseudospectrum, there is significantly more structure to be found. In fact, we show that $H$ is unitarily equivalent to a certain rotated harmonic oscillator, discussed in Section \ref{subsec.rot.ho}, and therefore information on the pseudospectrum and basis properties of eigenfunctions can be transferred directly to Swanson's model. Such a unitary equivalence is not a special property of this particular model; see Section \ref{subsec.el.q.op}.

In our particular case, a suitable unitary transform for $H$ reads
\begin{equation}\label{sw.U}
\begin{aligned}
\mathcal U &:= \mathcal U_1 \mathcal U_2 \mathcal U_3, \\
( \mathcal U_1 \psi)(x)&:= e^{-i \delta x^2 /2} \psi(x), 
\\
( \mathcal U_2 \psi)(x)&:= (\mathcal F^{-1} e^{-i \xi^2/(4 \delta)} \mathcal F\psi) (x),  
\\
( \mathcal U_3 \psi)(x)&:= (2 \delta)^{1/4} \psi((2\delta)^{1/2}x ), 
\end{aligned}
\end{equation}
where $\delta := (\omega+\alpha+\beta)^{1/2}(\omega-\alpha-\beta)^{-1/2}$ and $\mathcal F$ is the unitary Fourier transform 
\begin{equation}\label{eUnitary3}
\mathcal{F}u(\xi) = \frac{1}{\sqrt{2\pi}}\int_{\Real} e^{-ix\xi}\,u(x)\,\dd x, \quad u \in L^1(\Real) \cap L^2(\Real).
\end{equation}
The key steps are the relations for 
$\mathcal U_2^* \mathcal U_1^* x \mathcal U_1 \mathcal U_2$ 
and $\mathcal U_2^* \mathcal U_1^* p \mathcal U_1 \mathcal U_2$. Several straightforward manipulations and properties of $\mathcal F$ yield that for any $\varphi \in \mathscr{S}(\Real)$
\begin{equation}
\mathcal U_2^* \mathcal U_1^* x \mathcal U_1 \mathcal U_2 \varphi=
\left(x - \frac{i}{2 \delta} \frac{\dd}{\dd x}
\right)
\varphi,
\quad 
\mathcal U_2^* \mathcal U_1^* p \mathcal U_1 \mathcal U_2 \varphi= 
\left(
-\frac{i}{2} \frac{\dd}{\dd x} - \delta x
\right)
\varphi.
\end{equation}
The latter implies
\begin{equation}
\mathcal U_2^* \mathcal U_1^* H \mathcal U_1 \mathcal U_2 = \zeta  
\left(
-\frac{1}{2 \delta} \frac{\dd^2}{\dd x^2}  
+ 2\delta  \frac{\overline{\zeta}}{\zeta} x^2  
\right)
\end{equation}
with $\zeta := \sqrt{\omega^2 - (\alpha+\beta)^2} + i (\alpha -\beta)$. The additional rescaling $\mathcal{U}_3$ finally gives a multiple of the rotated harmonic oscillator:
\begin{equation}\label{Sw.reduct}
\mathcal U^* H \mathcal U = \zeta \left( -\frac{\dd^2}{\dd x^2}  
+ \frac{\overline{\zeta}}{\zeta} x^2  \right).
\end{equation}

The unitary equivalence shows that the eigenvectors of $H$ do not form a basis and the norms of the spectral projections grow as in \eqref{rot.osc.rate} with an appropriately chosen $\theta$.
Moreover, the numerically computed pseudospectra for~$H$ correspond to those for the rotated oscillator in Figure~\ref{Fig.HO-like} after the appropriate adjustment of parameters.
As in previous examples, the existence of non-trivial pseudospectra
makes~$H$ very different from any self-adjoint operator,
despite the reality of its spectrum 
and a formal similarity to a self-adjoint operator.
In particular, the spectrum is highly unstable 
under small perturbations.

\subsection{Elliptic quadratic operators}
\label{subsec.el.q.op}

To better understand and extend the reduction \eqref{Sw.reduct} applied to the gauged oscillator, we now discuss general operators which are quadratic in $(x,\dd/\dd x)$.  We begin with a quadratic symbol $q:\Real^2 \to \Com$,
\[
	q(x,\xi) := \alpha x^2 + 2\beta x\xi + \gamma \xi^2,
\]
where the $x$ variable represents the multiplication operator 
and the $\xi$ variable represents 
the self-adjoint momentum operator $-i \, \dd/\dd x$ 
defined on $W^{1,2}(\Real)$. 
Such a representation necessarily involves 
a choice for $x\xi$ between $x(\dd/\dd x)$ and $(\dd/\dd x)x$; 
in the quadratic case, the \emph{Weyl quantization} makes the choice
\[
	x\xi \mapsto 
\frac{1}{2}\left(x \Big(-i\frac{\dd}{\dd x} \Big) +  \Big(-i\frac{\dd}{\dd x} \Big) x\right).
\]
This choice ensures that real-valued $q$ lead to self-adjoint operators, in addition to other convenient properties.  (See, for instance, \cite[Chap.~7]{Dimassi-Sjostrand} or \cite[Sec.~18.5]{Hormander_3} for a far more general setting.)

We therefore arrive at the operator
\begin{equation}\label{eWeylDef}
	\begin{aligned}
	Q & := q^w(x,\xi) = \alpha x^2 -i \beta \left(x \frac{\dd}{\dd x}  + \frac{\dd}{\dd x} x \right) - \gamma \frac{\dd^2}{\dd x^2},
	\\ \Dom Q & := \{u \in \sii(\Real) \ \big| \ Qu \in L^2(\Real)\}.
	\end{aligned}
\end{equation}
The characterization of the domain as the graph closure of the restriction 
to $\mathscr{S}(\Real)$
or $C_0^\infty(\Real)$ 
may be found in \cite[pp.~425--426]{Hormander_1995}.

\paragraph{$\bullet$ Ellipticity}

It is natural to assume that $q(x,\xi)$ only vanishes at the origin:
\begin{equation}\label{eqEll1}
 	q(x,\xi) = 0 \implies (x,\xi) = (0,0).
\end{equation}
However, this is not sufficient to rule out degenerate behaviour of $Q$, and so one adds the assumption that
\begin{equation}\label{eqEll2}
 	q(\Real^2) \neq \Com.
\end{equation}
These conditions together assure us that there exists some nonzero complex number $\mu \in \Com$ for which 
\begin{equation}\label{rotated.ellipticity}
	\Re(\mu q(x,\xi)) \geq |(x,\xi)|^2
\end{equation}
and thus $\Re(\mu Q) = \frac{1}{2}(\mu Q + \overline{\mu}Q^*)$ acts like a harmonic oscillator.

If \eqref{eqEll1} holds but \eqref{eqEll2} fails, 
we find ourselves in a situation resembling 
that of a creation or annihilation operator \eqref{creation} squared: either 
$\dim \Ker(Q-z) = 2$ for all $z \in \Com$
or 
$\dim \Ran(Q-z)^\perp = 2$ 
for all $z \in \Com$; see \cite[Sec.~3.1]{Pravda-Starov_2008}. 
 This is precisely the situation which arises for the gauged oscillator, Section \ref{subsec.Swanson}, when~\eqref{Swanson.cond1} holds but~\eqref{Swanson.cond2} fails.

Under \eqref{eqEll1} and \eqref{eqEll2}, the spectral theory of the operator $Q$ can be deduced from the spectral theory of the matrix sometimes called the \emph{fundamental matrix}:
\begin{equation}\label{eFundamentalMatrix}
	F := \left(\begin{array}{cc} \beta & \gamma \\ -\alpha & -\beta\end{array}\right).
\end{equation}
It is shown in \cite[Prop.~3.3]{Sjostrand_1974} that
\begin{equation}\label{eQuadEigensystem}
	\begin{aligned}
	\sigma(F) & = \pm \lambda,
	\\ \Ker(F\mp \lambda) & = \operatorname{span}\left\{(1, a_\pm)\right\}, \quad & \pm \Im a_\pm > 0.
	\end{aligned}
\end{equation}
We note that, in dimension 1, we can identify $\mu = i/\lambda$ 
in \eqref{rotated.ellipticity} through choosing~$\lambda$ 
according to the signs of $\Im a_\pm$.  Taking this notation into account, we then have from \cite[Thm.~3.5]{Sjostrand_1974} that
\[
	\sigma(Q) = \big\{-i\lambda(2k+1) \ \big| \ k = 0, 1, 2, \dots\big\}.
\]

\paragraph{$\bullet$ Linear symplectic transformations}

One advantage of the Weyl quantization is that we may transform our symbols by composition with symplectic transformations; 
see \eg~\cite[Sec.~18.5]{Hormander_3} for a detailed accounting of the theory. In our simplified (linear, dimension one) setting, the set of real linear symplectic transformations is simply the set of a 2-by-2 matrices with real entries and determinant equal to one.  Any such matrix may be written \cite[Lem.~18.5.8]{Hormander_3} as a composition of matrices of the form
\begin{equation}\label{eCanonicalGenerators}
	G_b := \left(\begin{array}{cc} 1 & 0 \\ b & 1 \end{array}\right), \quad V_c := \left(\begin{array}{cc} c & 0 \\ 0 & 1/c \end{array}\right), \quad J := \left(\begin{array}{cc} 0 & 1 \\ -1 & 0\end{array}\right).
\end{equation}
with $b,c \in \Real$ and $c \neq 0$.

What is more, one may transform a quadratic symbol by composing with such a matrix by conjugating with an easily-understood unitary transformation on $L^2(\Real)$.  Specifically, we use multiplication by a complex Gaussian
\begin{equation}\label{eUnitary1}
	\mathcal{G}_bu(x) := e^{ibx^2/2}u(x),
\end{equation}
a scaling change of variables,
\begin{equation}\label{eUnitary2}
	\mathcal{V}_cu(x) := c^{-1/2}u(x/c),
\end{equation}
and the unitary Fourier transform~\eqref{eUnitary3}.
Using definition \eqref{eWeylDef}, it is then straightforward to check that
\[
	\begin{aligned}
		\mathcal{G}_b^* q^w(x,\xi)\mathcal{G}_b 
& = (q\circ G_b)^w(x,\xi),
		\\ \mathcal{V}_c^* q^w(x,\xi)\mathcal{V}_c 
& = (q\circ V_c)^w(x,\xi),
		\\ \mathcal{F}^* q^w(x,\xi)\mathcal{F} 
& = (q\circ J)^w(x,\xi).
	\end{aligned}
\]

One important example is the change of variables which gives us the correspondence between high-energy and semiclassical limits: for any $z \in \Com$,
\begin{equation}\label{eSemiclassicalScaling}
	q^w(x,\xi)-z = \mathcal{V}_{\sqrt{h}}^*\left(\frac{1}{h}(q^w(x,h \xi) - hz)\right)\mathcal{V}_{\sqrt{h}}.
\end{equation}
Therefore the regime with spectral parameter $rz$ as $r \to \infty$ for the operator $q^w(x,\xi)$ is unitarily equivalent to the regime with spectral parameter $z$ fixed for the operator $\frac{1}{h}q^w(x,h\xi)$ as $h = 1/r \to 0^+$.

\paragraph{$\bullet$ Reduction to rotated harmonic oscillator}

In \cite[Lem.~2.1]{Pravda-Starov_2007}, Pravda-Starov identifies a procedure for taking an elliptic quadratic form $q$ and finding $\mu \in \Com \backslash \{0\}$ and a real matrix $T$ with $\det T = 1$ for which
\begin{equation}\label{eKarelReduction}
	(q\circ T)(x,\xi) = \mu\left((1+i\lambda_1)x^2 + (1+i\lambda_2)\xi^2\right)
\end{equation}
for $\lambda_1, \lambda_2 \in \Real$. Applying a scaling like \eqref{eSemiclassicalScaling} and scaling $\mu$ allows us to assume that the coefficients of $x^2$ and $\xi^2$ have the same modulus.  It is then evident that the resulting symbol is a multiple of that of a rotated harmonic oscillator \eqref{HO.rotated}.

If one wishes only to identify the parameters of the rotated harmonic oscillator involved, one may appeal to the spectral decomposition of the fundamental matrix and the growth of the spectral projections. An application of Corollary~1.7 in \cite{Viola_2013} in terms of the eigensystem \eqref{eQuadEigensystem} shows that the norm of the spectral projections $P_k$ for the eigenvalues $(2k+1)\lambda/i$ obey the asymptotics
\[
 	\lim_{k\to\infty}\frac{\log \|P_k\|}{k} = \frac{1}{2}\log \frac{1+|c_+|}{1-|c_+|}, \quad c_+ = -\frac{a_+ - \overline{a_-}}{a_+ - a_-}.
\]
From \eqref{rot.osc.rate}, we see that this uniquely identifies a rotated harmonic oscillator with $\theta \geq 0$.  The multiplicative factor, in turn, is determined by the ground state energy, that is, the eigenvalue corresponding to $k = 0$.  We arrive at the following proposition.

\begin{Proposition}\label{Prop.QuadraticIdentify}
 	Let $Q$ be any quadratic operator as in \eqref{eWeylDef} with symbol $q$.  Assume that $q$ is elliptic as in \eqref{eqEll1} and \eqref{eqEll2}, and therefore let the eigensystem of the fundamental matrix of $q$ be as in \eqref{eQuadEigensystem}.  Let $\theta \in [0, \pi/2)$ be determined by
	\[
	 	\sin \theta = \left|\frac{a_+ - \overline{a_-}}{a_+ - a_-}\right|.
	\]
	
	Then $Q$, as an unbounded operator on $L^2(\Real)$, is unitarily equivalent to
	\[
	 	\lambda \left(-e^{-i\theta}\frac{\dd^2}{\dd x^2} + e^{i\theta}x^2\right).
	\]
\end{Proposition}

\paragraph{$\bullet$ Higher dimension}

The extension of the spectral and pseudospectral theory to elliptic quadratic operators in higher dimension is well-developed but not complete.  We content ourselves with a brief description and references.

The Weyl quantization of a quadratic form 
\[
	q(x,\xi): \Real^{2d} \to \Com
\]
in dimension $d \geq 2$ also associates the variable $x_j$ with multiplication by $x_j$, associates the variable $\xi_j$ with $-i\partial/\partial x_j$, and resolves the problem of commutativity (in the quadratic case) by taking an average:
\[
	(x_j\xi_j)^w u(x) = \frac{1}{2i}\left(x_j\frac{\partial}{\partial x_j}u(x) + \frac{\partial}{\partial x_j}(x_j u(x))\right).
\]
The ellipticity hypothesis is simpler, since in dimension 2 or greater, \eqref{eqEll1} implies~\eqref{eqEll2}, shown in \cite[Lem.~3.1]{Sjostrand_1974}.

Under the assumption \eqref{eqEll1}, the spectrum of $Q = q^w(x,\xi)$ is a lattice determined by eigenvalues of the matrix corresponding to the fundamental matrix \eqref{eFundamentalMatrix}; the formula is given in \cite[Thm.~3.5]{Sjostrand_1974}.

It was recently shown in \cite{Caliceti-Graffi-Hitrik-Sjostrand_2012} that, 
regardless of dimension, an elliptic quad\-rat\-ic form obeying 
a $\PT$-symmetry condition is \emph{formally} similiar to 
a self-adjoint operator if and only if the spectrum is real 
and the fundamental matrix is diagonalizable.  
This similarity is only enacted eigenspace by eigenspace, and the authors observe that the pseudospectral considerations prevent the similarity transformations from being bounded with bounded inverse on $L^2(\Real^d)$.

In \cite{Pravda-Starov_2008}, Pravda-Starov conducts a complete study of the semiclassical pseudospectrum for non-normal elliptic quadratic operators.  It is shown in \cite[Sec.~3.2]{Pravda-Starov_2008} that, for a non-normal operator, the bracket condition is violated everywhere in the interior of the range of the symbol and therefore exponentially accurate pseudomodes exist.

Conversely, exponential upper bounds for the resolvent are proven in \cite{Hitrik-Sjostrand-Viola_2013}, except that exponential growth $C^{1/h}$ may need to be replaced by the more rapid growth $(C/h)^{C/h}$ when the fundamental matrix contains Jordan blocks.

Finally, the associated spectral projections for a non-normal quadratic operator were shown in \cite{Viola_2013} to usually increase at an exponential rate, though there are degenerate situations such as when Jordan blocks are present in the fundamental matrix.

\subsection{Numerical computation of JWKB solutions}
\label{subsec.numer.wkb}

The proof of Theorem~\ref{Thm.DSZ} proceeds by creating pseudomodes as JWKB (Jeffreys-Wentzel-Kramers-Brillouin) approximations for which $(H-z)u \approx 0$.  Focusing on the one-dimensional Schr\"odinger case, these functions can be expressed using a few elementary operations, including differentiation and integration; see \eqref{WKB.Full}, \eqref{WKB.Phase}, \eqref{WKB.a0}, and \eqref{WKB.aj}.  The software package Chebfun \cite{Chebfun} allows us to easily compute these pseudomodes with high accuracy.

In Figure \ref{Fig.Modes.RHO} we compare a JWKB pseudomode for the semiclassical rotated harmonic oscillator
\begin{equation}\label{numerics.rot.ho}
	H_h = -h^2e^{-i\pi/4}\frac{\dd^2}{\dd x^2} + e^{i\pi/4}x^2
\end{equation}
discussed in Section \ref{subsec.rot.ho} with $z = e^{-i\pi/4}(1/2+i)$, $h = 2^{-5}$, and 
\[
	u(x;h) = \chi(x)e^{i\varphi(x)/h}\sum_{j=0}^6 h^j a_j(x).
\]
The plot on the left is of the real part (red) and the imaginary part (blue) of the pseudomode, and the plot on the right is of the image, $(H_h-z)u$.  One may compute that
\[
	\frac{\|(H_h-z)u\|}{\|u\|} \approx 2.5041 \times 10^{-4}.
\]
One can clearly see the contributions from the gradient of the support of the cutoff function, which is $[0.2,0.4] \cup [1.6,1.8]$.
\begin{figure}
\begin{center}
\includegraphics[width=0.99\textwidth]{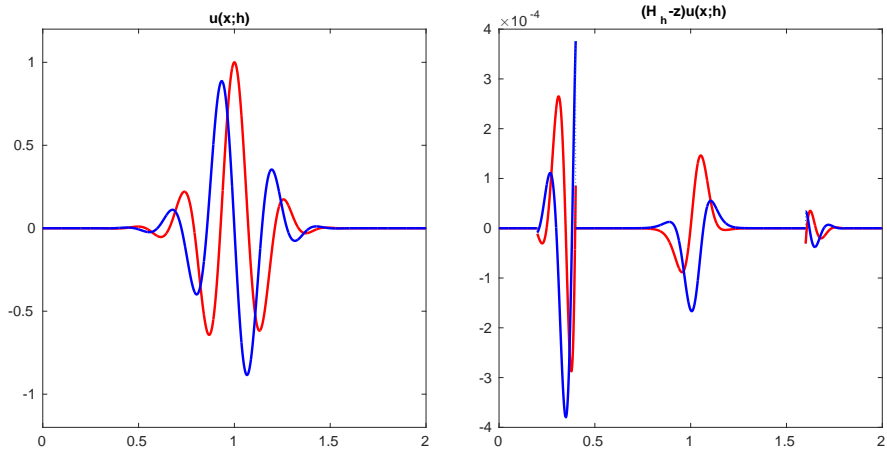}
\end{center}
\caption{Pseudomode (left) and image (right) for semiclassical rotated harmonic oscillator. The red curve is the real part, and the blue the imaginary.}\label{Fig.Modes.RHO}
\end{figure}

If one studies instead the semiclassical rescaling of the shifted harmonic oscillator
\begin{equation}\label{numerics.shift.ho}
	H_h = -h^2\frac{\dd^2}{\dd x^2} + x^2 + 2ih^{1/2}x - h
\end{equation}
discussed in Section \ref{subsec.shift.ho} we have noticeable but less accurate pseudomodes, with the principal error given on the support of the cutoff function.  In Figure \ref{Fig.Modes.SHO}, one has a similar JWKB solution and its image with $h = 2^{-8}$, $z = 2-h+2ih^{1/2}$, and 
\[
	u(x;h) = \chi(x)e^{i\varphi(x)/h}\sum_{j=0}^2 h^j a_j(x).
\]
(We may see numerically that there is practically no difference in norms when taking one, two, or ten terms in the expansion.)  Since $h$ is very small, the JWKB function oscillates quite rapidly, and since the decay of $e^{i\varphi/h}$ is comparatively weak, 
the principal error comes from the cutoff function, whose gradient is again supported on $[0.2,0.4] \cup [1.6,1.8]$. We have here
\[
	 \frac{\|(H_h-z)u\|}{\|u\|} \approx 2.0290 \times 10^{-3}.
\]
\begin{figure}[h!]
\begin{center}
\includegraphics[width=0.99\textwidth]{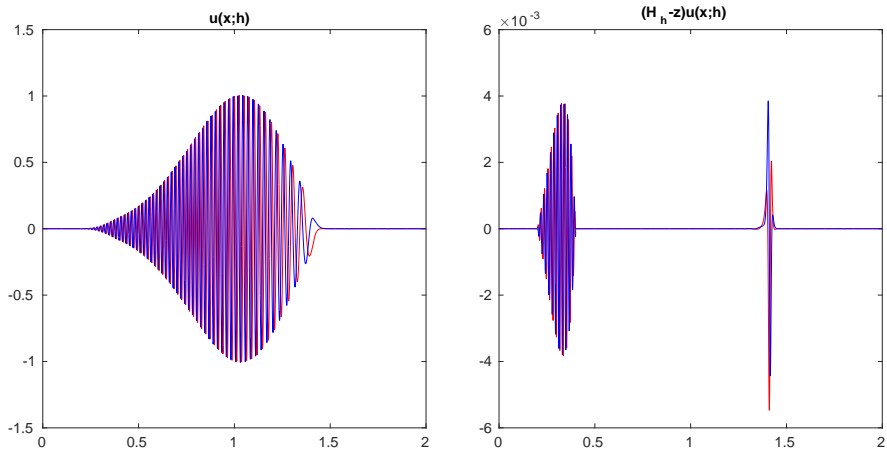}
\end{center}
\caption{Pseudomode (left) and image (right) for semiclassical rescaling of shifted harmonic oscillator. The red curve is the real part, and the blue the imaginary.}\label{Fig.Modes.SHO}
\end{figure}

We can then compare the accuracy of the $\sii(\Real)$-normalized pseudomodes by plotting $\|(H_h-z)u\|$ versus $1/h$ for a variety of $h$, presented in Figure \ref{Fig.Accuracy.Compare}. On the left, we have the norms for the semiclassical rotated harmonic oscillator \eqref{numerics.rot.ho} at $z = 1+4i$, and on the right, we have those for the rescaled shifted harmonic oscillator \eqref{numerics.shift.ho} at $z = 2-h+2ih^{1/2}$. We can observe that the norm ratios for the pseudomodes for the rotated harmonic oscillator decrease like $\exp(-c/h)$ while those for the shifted harmonic oscillator decrease more slowly.

\begin{figure}[h!]
\begin{center}
\includegraphics[width=0.99\textwidth]{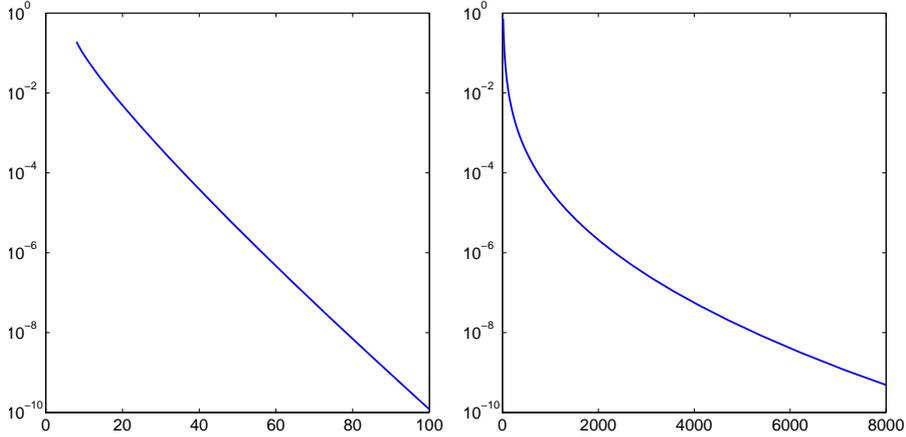}
\end{center}
\caption{$\|(H_h-z)u\|$ as a function of $1/h$ for normalized pseudomodes for rotated harmonic oscillator (left) and shifted harmonic oscillator (right).}\label{Fig.Accuracy.Compare}
\end{figure}

\appendix
\section{Existence proofs for pseudomodes}\label{App.proofs}

For the interested reader, we include detailed proofs of Theorem~\ref{Thm.DSZ} in the case of a Schr\"odinger operator and of Theorem~\ref{Thm.Shifted}.  In these proofs, the constant $C>0$ may change from line to line.  Furthermore, we understand semiclassical statements involving $h$ to only hold for $h \in (0, h_0]$ for some $h_0 > 0$; so long as $h_0$ changes only finitely many times in the proof, we are allowed to make conclusions ``for $h$ sufficiently small'' in our theorems.

\subsection{Proof of special case of Theorem~\ref{Thm.DSZ}}
We restrict our attention to the case
\[
H_h = -h^2\frac{\dd^2}{\dd x^2} + V(x).
\]
The symbol of $H_h$ is
\begin{equation}\label{Symbol.Schro}
	f(x,\xi) = \xi^2 + V(x),
\end{equation}
and
\[
	\frac{1}{2i}\left\{f,\bar{f}\right\} = -2\xi\Im V'(x).
\]
Therefore $z \in \Lambda$, defined in \eqref{semiclass.pseudospec}, if and only if there exists $(x_0, \xi_0) \in \Real^2$ with $z = \xi_0^2 + V(x_0)$, $\Im V'(x_0) \neq 0$, and $-\xi_0 \Im V'(x_0) > 0$.  Equivalently, since we may choose the sign of $\xi_0$, there exists some $x_0 \in \Real$ where $\Im V'(x_0) \neq 0$ and $z-V(x_0)$ is a positive real number.
We only need to assume that $V(x)$ is analytic in a neighborhood of $x_0$.

After a translation, we may assume that $x_0 = 0$.  We seek a JWKB approximation (see for instance \cite[Chap.~2]{Dimassi-Sjostrand}) to a solution of $(H_h-z)u = 0$ of the form
\begin{equation}\label{WKB.Full}
	u(x;h) = e^{i\varphi(x)/h}\sum_{j=0}^{N(h)} h^j a_j(x)
\end{equation}
for $a_j(x)$ analytic near $x_0 = 0$. The strategy is to choose the phase function such that conjugation by the multiplication operator $e^{-i\varphi(x)/h}$ reduces $H_h$ to a transport equation plus an error in $h$. The functions $a_j$ may be found iteratively and then $N(h)$ may be chosen to give an accurate local solution. The quasimode will then be obtained by multiplying $u(x;h)$ by a fixed cutoff function $\chi$ localizing to a neighborhood of $x_0=0$.  An important difference making the JWKB theory for non-self-adjoint operators somewhat simpler is that the phase function $\varphi(x)$ has a significant imaginary part. This allows for multiplication by cutoff functions with small errors, a technique which is generally not available for self-adjoint operators where $\varphi$ is real-valued. 

We require that the phase function $\varphi(x)$ satisfies the eikonal equation
	\[
	 f(x,\varphi'(x))-z = 0
	\]
	for $f$ from \eqref{Symbol.Schro}.
	Clearly this implies that $\varphi'(x) = \pm\sqrt{z-V(x)}$.
Since $z-V(0) > 0$, this function is analytic in a neighborhood of $0 \in \Com$.

	We allow the sign to be determined by the bracket condition
	\begin{equation}\label{bracket.schro}
	 	\frac{1}{2i}\{f,\bar{f}\}(x,\xi) = -2 \Im V'(x)\xi> 0.
	\end{equation}
	Applying this to $(x,\xi) = (0, \varphi'(0))$ indicates that the sign of $\varphi'(0)$ should be taken to be the opposite of the sign of $\Im V'(0)$.  Alternately, the importance of this choice of sign may be seen by observing that 
	\begin{equation}\label{e.ddphi}
	 	\varphi''(x) = -\frac{V'(x)}{2 \varphi'(x)}
	\end{equation}
	and thus our choice is made so that $\Im \varphi''(0) > 0$, which means that $e^{i\varphi(x)/h}$ has rapid decay away from $x = 0$.  We arrive at the formula
\begin{equation}\label{WKB.Phase}
	\varphi(x) = -\operatorname{sgn}(\Im V'(x_0))\int_0^x\sqrt{z-V(y)}\,\dd y .
\end{equation}
	We may then check that
	\[
	 	e^{-i\varphi/h}(H_h-z)e^{i\varphi/h} = \frac{2h}{i}(\varphi'\frac{\dd}{\dd x}+\frac{1}{2}\varphi'')-h^2 \frac{\dd^2}{\dd x^2}.
	\]
	So long as $\{a_j\}_{j=0}^\infty$ satisfy the transport equations
	\[
		\varphi'(x)a_0'(x) + \frac{1}{2}\varphi''(x) a_0(x) = 0
	\]
	and
	\[
	 	\varphi'(x)a_j'(x) + \frac{1}{2}\varphi''(x) a_j(x) = \frac{i}{2}a_{j-1}''(x), \quad j=1, 2, \dots,
	\]
	we have
	\begin{equation}\label{eWKBOutput}
		e^{-i\varphi/h}(H_h-z)e^{i\varphi/h}\left(\sum_{j=0}^N h^j a_j\right) = -h^{N+2}a_N''.
	\end{equation}
	We are free to choose $a_0(0) = 1$ and $a_j(0) = 0$ for all $j > 0$.  Using the integrating factor $\exp(\int_0^x \varphi''(y)/(2\varphi'(y))\,\dd y) = C\sqrt{\varphi'(x)}$ immediately gives that
	\begin{equation}\label{WKB.a0}
	 	a_0(x) = \frac{\sqrt{\varphi'(0)}}{\sqrt{\varphi'(x)}}
	\end{equation}
	and that, for $j > 0$,
	\begin{equation}\label{WKB.aj}
	 	a_{j+1}(x) = \frac{1}{\sqrt{\varphi'(x)}}\int_0^x \frac{i a_j''(y)}{2\sqrt{\varphi'(y)}}\,dy.
	\end{equation}
	We note that, in a sufficiently small neighborhood of zero in the complex plane, $\varphi'$ may be extended to an analytic function which is bounded away from zero, and therefore each $a_j$ is certainly analytic on that neighborhood of zero.

	We now consider bounds on the functions $a_j$.  As in Example 1.1 of \cite{Sjostrand_SAM}, we will show that the $a_j$ obey the estimates
	\begin{equation}\label{ejToThej}
	 	|a_j(z)| \leq C_1^{j+1}j^j
	\end{equation}
	for some $C_1 > 0$ and all $z$ in a neighborhood of the origin. A sequence of functions satisfying these estimates is said to be a \emph{formal analytic symbol}. Once these bounds are established, we may define the $h$-dependent function
	\begin{equation}\label{eaDef}
	 	a(z;h) = \sum_{0 \leq j \leq (eC_1h)^{-1}} h^ja_j(z),
	\end{equation}
	summing over a collection of $j$ chosen such that
	\begin{equation}\label{eeToThej}
	 	|h^ja_j(z)| \leq C_1(C_1hj)^j \leq C_1 e^{-j}.
	\end{equation}
	Since $\{e^{-j}\}_{j\geq 0}$ is summable, we will therefore have that $\{a(z;h)\}_{0< h \leq h_0}$ is a uniformly bounded collection of analytic functions on the set where \eqref{ejToThej} holds.

	The natural norm to use here for analytic functions is the supremum norm,
so for $K \subseteq \Com$ we write
	\[   
	 	\|g\|_{K} = \sup_{z \in K} |g(z)|.
	\]
	For balls in the complex plane centered at zero, we use the notation
	\[
	 	B(R) = \left\{z \in \Com \ \big| \ |z| < R\right\}.
	\]
	Fix $R_0 > 0$ such that, on $B(R_0)$, the phase function $\varphi$ is analytic, the modulus of the derivative $|\varphi'|$ is bounded from above and below, and $\Im \varphi''(x) > 1/C$ for some $C > 0$.

	Cauchy's estimates for the second derivative of a analytic bounded function $g$ defined on $B(R)$ read
	\begin{equation}\label{eCauchy2nd}
	 	|g''(z)| \leq \frac{2\|g\|_{B(R)}}{(R-|z|)^{2}}.
	\end{equation}
	We integrate the estimates applied to $a_j''$ to obtain bounds for $a_{j+1}$:
	\begin{equation}\label{eIntegrateCauchy}
		\begin{aligned}
	 	|a_{j+1}(z)| &= \left|\frac{1}{\sqrt{\varphi'(z)}}\int_0^z \frac{i a_j''(\zeta)}{2\sqrt{\varphi'(\zeta)}}\,d\zeta\right| 
		\\ & \leq \|(\varphi')^{-1}\|_{B(R)} \int_0^{|z|} \frac{\|a_j\|_{B(R)}}{(R-t)^2}\,\dd t
		\\ & \leq \|(\varphi')^{-1}\|_{B(R)}\|a_j\|_{B(R)}\left(\frac{1}{R-|z|}-\frac{1}{R}\right) 
		\\ & = \frac{|z|}{R(R-|z|)}\|(\varphi')^{-1}\|_{B(R)}\|a_j\|_{B(R)}.
		\end{aligned}
	\end{equation}
	The estimate for $|a_{j+1}(z)|$ is stronger than the usual Cauchy's estimate for the first derivative when $z$ is near zero, which corresponds to having set $a_{j+1}(0) = 0$.

	To obtain the estimates \eqref{ejToThej} on $B(R_0/2)$, we fix $j > 0$ and iterate \eqref{eIntegrateCauchy} on balls of radius
	\[
	 	R_k = \left(1-\frac{k}{2j}\right)R_0, \quad k = 0,\dots,j-1.
	\]
	When $|z| \leq R_{k+1}$, we have
	\[
		\frac{|z|}{R_k(R_k-|z|)} \leq \frac{|z|}{R_k(R_k - R_{k+1})} = \frac{|z|}{(R_0/2)(R_0/2j)} \leq \frac{4j|z|}{R_0^2},
	\]
	since $R_k > R_0/2$ when $k < j$.  Therefore we may bound $a_{k+1}$ on the disc of radius $R_{k+1}$ using a bound for $a_k$ on the disc of radius $R_k$ and \eqref{eIntegrateCauchy}:
	\begin{equation}\label{eIteratedEstimate}
		\begin{aligned}
	 	\|a_{k+1}\|_{B(R_{k+1})} & \leq \frac{|z|}{R_k(R_k-|z|)}\|(\varphi')^{-1}\|_{B(R_k)}\|a_k\|_{B(R_k)} 
		\\ & \leq \frac{4j|z|}{R_0^2}\|(\varphi')^{-1}\|_{B(R_0)}\|a_k\|_{B(R_k)}.
		\end{aligned}
	\end{equation}
	Therefore, when $j > 0$, we take the product of the estimates \eqref{eIteratedEstimate} for $k=0,\dots,j-1$ to obtain
	\begin{equation}\label{eModz}
	 	\|a_j\|_{B(R_0/2)} \leq \|a_0\|_{B(R_0)}\left(C_2 |z| j\right)^j, \quad C_2 = \frac{4}{R_0^2}\|(\varphi')^{-1}\|_{B(R_0)}.
	\end{equation}
	The estimate \eqref{ejToThej} on $B(R_0/2)$ immediately follows, with
	\begin{equation}\label{eC1phiprime}
	 	C_1 = \max\left(\|a_0\|_{B(R_0)}, \frac{2}{R_0}\|(\varphi')^{-1}\|_{B(R_0)}\right).
	\end{equation}

	Having established estimates for $a(z;h)$ when $z \in B(R_0/2)$, let 
$\chi \in C_0^\infty(\Real)$ 
be equal to one in a neighborhood of $0 \in \Real$ and have support in a compact subset of the interval $(-R_0/2, R_0/2)$.  We then define our pseudomode as
	\[
	 	u(x;h) = e^{i\varphi(x)/h}\chi(x)a(x;h),
	\]
	with $a(x;h)$ defined in \eqref{eaDef}.
	
	We then estimate the $L^2(\Real)$ norm
	\begin{equation}\label{ePuTriangle}
		\|H_h-z)u(x;h)\| \leq \|\chi (H_h-z)e^{i\varphi/h}a\| + \|[(H_h-z),\chi]e^{i\varphi/h}a\|
	\end{equation}
	as follows.  First, we recall that we have chosen $R_0$ such that $\Im \varphi''(x) > 1/C_3$ for some $C_3 > 0$.  Since $\varphi(0) = 0$ and $\varphi'(0)$ is real, we therefore have that
	\begin{equation}\label{eDecaysuppchi}
	 	|e^{i\varphi(x)/h}| \leq \exp\left(-\frac{1}{2C_3h}x^2\right), \quad \forall x \in \operatorname{supp}\chi.
	\end{equation}
	Since $|e^{-i\varphi/h}| \geq 1$ on $\operatorname{supp}\chi$, we may multiply by $e^{-i\varphi/h}$ and use \eqref{eWKBOutput} to obtain
	\[
		\|\chi H_h e^{i\varphi/h}a\| \leq \|\chi e^{-i\varphi/h} H_h e^{i\varphi/h}a\| = \|h^{N+2}a_N''(x)\chi(x)\|,
	\]
	with $N = N(h) = \lfloor(eCh)^{-1}\rfloor$.  Cauchy's estimates \eqref{eCauchy2nd} along with \eqref{eeToThej} show that $h^{N+2}a_N''(x) \leq Ce^{-1/(Ch)}$ for $C > 0$ independent of $h$ and all $x \in \operatorname{supp}\chi$. We therefore have, for some $C > 0$, the estimate
	\[
		\|\chi H_h e^{i\varphi/h}a\| \leq Ce^{-1/(Ch)}.
	\]

	In the commutator of $(H_h-z)$ and $\chi$, only the derivatives in $H_h$ play a role.  We compute 
	\begin{equation}\label{WKB.Commutator}
		\begin{aligned}
	 		\left[(H_h-z), \chi \right] e^{i\varphi/h}a &= \left[-h^2\frac{\dd^2}{\dd x^2},\chi\right]e^{i\varphi/h}a 
			\\ 
			& = -h^2e^{i\varphi/h}\left(\chi'' a + 2\chi'\left(a' + \frac{i\varphi'}{h}a\right)\right).
		\end{aligned}
	\end{equation}
	On $\operatorname{supp}(\chi)$, we have uniform bounds on $a$ by \eqref{eeToThej} and therefore on $a'$ by Cauchy's estimates.  As before, $\varphi'$ is controlled by the choice of $R_0$.  Exponential decay comes from the fact that $\operatorname{supp}(\chi')$ avoids a neighborhood of $0$: by \eqref{eDecaysuppchi}, we have that 
	\begin{equation}\label{eDecaygradchi}
		|e^{i\varphi(x)/h}| \leq e^{-1/(Ch)}, \quad \forall x \in \operatorname{supp}(\chi').
	\end{equation}
	Therefore the second term in \eqref{ePuTriangle} is also exponentially small in $1/h$.

	Having proven that both terms in \eqref{ePuTriangle} are exponentially small, the proof is complete upon showing that $u(x;h)$ is not exponentially small.  Intuitively, this is clear from the choice of $\varphi$ and that $a_0(0) = 1$ and $a_j(0) = 0$ for $j > 0$, from which we know that $u(x;h)$ resembles $e^{-\varphi''(0)x^2/(2h)}$ in a small neighborhood of zero.  Formally, since \eqref{eModz} gives $|h^ja_j(z)| \leq C(C|z|)^j$ for $|z| < R_0/2$ and $0 < j \leq N(h) = (eC_1h)^{-1}$, we have for some $r_0 > 0$ sufficiently small the estimate
	\[
	 	\left\| \sum_{j=1}^{N(h)} h^ja_j(z)\right\|_{B(r)} \leq Cr, \quad 0 < r \leq r_0.
	\]
	Since $a_0(z)$ is close to $1$ and $\Im \varphi(z)$ is close to 
$\frac{1}{2}\Im \varphi''(0)z^2$ when $z$ is close to $0$, 
we can consider $r$ sufficiently small and fixed to obtain
	\[
	 	\|u(x;h)\| \geq \|u(x;h)\|_{L^2((-r,r))} \geq \frac{1}{C}\left(\int_{-r}^r \exp\left(\frac{1}{Ch}x^2\right)\,dx\right)^{1/2} \geq \frac{1}{C}h^{1/4}
	\]
	when $h$ is sufficiently small.

	Since we have shown that $\|H_hu(x;h)\| \leq Ce^{-1/(Ch)}$ 
and $\|u(x;h)\| \geq h^{1/4}/C$, this completes the proof of the theorem in this special case.
\qed

\paragraph{$\bullet$ Uniformity on compact sets.}
We also remark that the exponential resolvent growth may generally be made uniform on compact subsets of $\Lambda$, the interior of the semiclassical pseudospectrum defined in \eqref{semiclass.pseudospec}.  In the case of the Schr\"odinger operator, for any $z \in \Lambda$ we may take $x_0$ with $\Im V(x_0) = \Im z$ and define the phase function
\[
	\varphi(x) = \pm \int_{x_0}^x \sqrt{z-V(y)}\,\dd y,
\]
with the sign chosen so that $\Im \varphi''(x_0) > 0$.  

The exponentially rapid resolvent growth then follows from having $C > 0$ and $R_1 > R_0 > 0$ for which the estimates
\begin{equation}\label{e.unif.1}
	|x - x_0| < R_1 \implies \frac{1}{C} \leq |\varphi'(x)| \leq C
\end{equation}
and
\begin{equation}\label{e.unif.2}
	R_0 < |x-x_0| < R_1 \implies \Im \varphi(x) \geq \frac{1}{C}
\end{equation}
hold: the former gives \eqref{ejToThej} by way of \eqref{eC1phiprime} and the latter gives \eqref{eDecaygradchi}, which are together sufficient to prove exponential growth of the resolvent.

The condition $\Im V'(x_0) \neq 0$ means that $x_0$ may be chosen locally as a continuous function of $\Im z$; it is then a simple matter to verify that \eqref{e.unif.1} holds with uniform constants in a neighborhood of $z \in \Lambda$.  Local uniformity of \eqref{e.unif.2} then follows from $\Im \varphi''(x_0) > 0$.

\subsection{Proof of Theorem~\ref{Thm.Shifted}}
As usual, we may make the change of variables $y = h^{-1/2}x$, arriving at an operator unitarily equivalent to $h H_h$:
\[
	\tilde H_h = -h^2\frac{\dd^2}{\dd y^2} + y^2 + 2ih^{1/2}y 
  - h \simeq h H_h.
\]
We will let $h^{-1} = \Re z$ so that
\[
 	H_h-z \simeq h^{-1}(\tilde H_h-(1+it)), \quad t = \frac{\Im z}{\Re z}.
\]

In this case, we have a symbol
\[
	\tilde{f}(y,\eta) = \eta^2 + y^2 + 2ih^{1/2}y -h,
\]
and so we cannot directly apply the results of \cite{Dencker-Sjostrand-Zworski_2004}.  We can, however, adapt the previous proof.

We can see that $1+it = \tilde{f}(y_0,\eta_0)$ for some $(y_0,\eta_0) \in \Real^2$ if and only if
\[
	1+h = \eta_0^2 + y_0^2, \quad t = 2ih^{1/2}y_0,
\]
implying that $|y_0| \leq 1+h$ and 
\[
	|t| \leq 2h^{1/2}(1+h).
\]

The eikonal equation $\tilde{f}(y, \varphi'(y))=0$ is then solved by integrating
\begin{align*}
	\varphi'(y) & = -\sqrt{(1+h+it)-ih^{1/2}y -y^2} 
	\\ & = -(1+h-y^2)\left(1+\frac{i(t-2h^{1/2}y)}{1+h-y^2} + \mathcal{O}\left(\left(\frac{t-2h^{1/2}y}{1+h-y^2}\right)^2\right)\right).
\end{align*}
As usual, we choose the sign of the square root to satisfy a bracket condition like \eqref{bracket.schro}.

To ensure that, on a uniform neighborhood of $y_0$, the phase $\varphi'(y)$ is analytic and there exists $C > 0$ for which $1/C \leq |\varphi'(y)| \leq C$, we must assume that $1+h-y^2$ is bounded away from zero by a constant.  We therefore assume that $|y_0| < 1-\eps$ for some $\eps > 0$.  We remark that this is connected to assuming that $1+it$ is in the interior of the range of the symbol. The scaling argument shows that this assumption is equivalent to the hypothesis that
\[
	\Im z \leq 2\left(1-\eps+(\Re z)^{-1}\right)\sqrt{\Re z}.
\]
The term $(\Re z)^{-1}$ is negligible as $\Re z \to \infty$.

Note that
\[
 	\varphi''(y) = \frac{-ih^{1/2} - y}{\varphi'(y)}
\]
and so, since we chose the sign of the square root to have $\varphi'(y_0) < 0$, we have
\[
 	\frac{1}{C}\sqrt{h} \leq \Im \varphi''(y) \leq C\sqrt{h}
\]
on a sufficiently small but fixed neighborhood. 

We may then construct exponentially accurate approximations to a solution of
\[
 	e^{-i\varphi/h} \tilde H_h e^{i\varphi/h}a(x;h) = 0
\]
exactly as in the case of the semiclassical Schr{\"o}dinger operator above.  Taking a cutoff function $\chi$ supported near $y_0$ and writing
\[
 	u(y;h) = \chi(y)e^{i\varphi(y)/h}a(y;h)
\]
as before, we have the same argument for exponential smallness except 
where we commute $\tilde H_h$ past the cutoff function $\chi$.  Because $\Im \varphi(y)$ increases more slowly, we only have
\[
 	|e^{i\varphi(y)/h}| \leq e^{-1/(Ch^{1/2})}, \quad y \in \operatorname{supp}(\chi').
\]
We arrive at 
\[
 	\|\tilde H_h u(y;h)\|_{L^2} \leq Ce^{-1/(Ch^{1/2})}
\]
and
\[
 	\|u(y;h)\|_{L^2} \geq \frac{1}{C}h^{1/8}
\]
for $h$ sufficiently small and positive.  Conjugating by the change of variables $y = (\Re z)^{1/2}x$ with which we began and ignoring harmless powers of $h = (\Re z)^{-1}$ proves the theorem.
\qed

\section{Spectral projections of the rotated oscillator}
\label{app.Pk}

Asymptotics for the norms of the spectral projections of the rotated oscillator may be found using an established integral formula involving the Hermite functions, pointed out in \cite{Bagarello-2010-374}, and asymptotics of the Legendre polynomials. This approach is analogous to one applied to the shifted harmonic oscillator in \cite[Sec.~2]{Mityagin-2013}, and it simplifies and sharpens the result of \cite{Davies-Kuijlaars_2004}.

The eigenfunctions of $H$ defined in \eqref{HO.rotated} can be written explicitly through a complex scaling of Hermite functions:
\begin{equation}\label{rs.HO.EF}
\begin{aligned}
H \left(h_k(e^{i \theta/2}x)\right)  & = (2k+1) h_k(e^{i \theta/2}x), \\
\end{aligned}
\end{equation}
where $h_k$ denote (normalized) Hermite functions. 
The eigenfunctions of the adjoint $H^*$  are obtained by complex conjugation, and it is easy to verify 
the biorthonormal relation
\begin{equation}\label{hk.bio}
\langle h_k(e^{i \theta/2}x), h_l(e^{-i \theta/2}x) \rangle = \delta_{kl}.
\end{equation}
One may show that the eigenfunctions are complete in $L^2(\Real)$, the corresponding eigenvalues are algebraically simple, and there are no other points in the spectrum, \cf~\cite{Davies_1999a}.
Consequently, the spectral projections~$P_k$ of $H$ can be written as
\begin{equation}
\begin{aligned}
P_k & = h_k(e^{i \theta/2}x) \langle h_k(e^{-i \theta/2}x), \cdot \rangle. 
\end{aligned}
\end{equation}
The Cauchy-Schwarz inequality, the biorthonormal relation \eqref{hk.bio}, and symmetries of the Hermite functions can be used to show that  
\begin{equation}
\begin{aligned}
\|P_k\| &= \| h_k(e^{i \theta/2}x)\| \|h_k(e^{-i \theta/2}x)\| = \| h_k(e^{i \theta/2}x)\|^2. 
\end{aligned}
\end{equation}
The resulting norms can be calculated explicitly. As pointed out by F. Bagarello in \cite{Bagarello-2010-374},  
the formula \cite[Eq.~2.20.16.2, p.~502]{Prudnikov-1986-2} 
\begin{equation}
\begin{aligned}
& \int_0^{\infty} e^{-a x^2} H_k(bx) H_k(cx) \dd x 
\\& \quad = \frac{2^{k-1} k! 
\sqrt{\pi}}{a^{(k+1)/2}} (b^2 + c^2 -a)^{k/2} \mathcal{P}_k 
\left(
\frac{bc}{\sqrt{a(b^2+c^2-a)}}
\right),
\end{aligned}
\end{equation}
which is valid if $\Re a >0$ and where $\mathcal{P}_k$ are the Legendre polynomials, 
yields
\begin{equation}\label{Pk.norm}
\|P_k\| = \frac{1}{(\cos \theta)^{1/2}} \mathcal{P}_k \left( \frac{1}{\cos \theta}\right).
\end{equation}
The final result comes from the asymptotic behaviour of $\mathcal{P}_k(x)$, \cf~the Laplace-Heine formula \cite[Thm.~8.21.1]{Szego-1959} or its generalization \cite[Thm.~8.21.2]{Szego-1959} which provides further terms: 
\begin{equation}
\begin{aligned}
\|P_k\| & = 
\frac{1}{\sqrt{2\pi k |\sin \theta|}} 
\left(
\frac{1 + |\sin \theta|}{\cos \theta}
\right)^{k+1/2} 
\left(1 + o(1)  \right).
\end{aligned} 
\end{equation}
We note that this exponential factor agrees with \eqref{rot.osc.rate} and with \cite{Davies-Kuijlaars_2004}, \cite{Henry}, following a simple computation in \cite[Ex.~3.6]{Viola_2013}.

\subsection*{Acknowledgment}
%
The work has been partially supported
by the project RVO61389005, the GACR grant No.\ 14-06818S
and the MOBILITY project 7AMB12FR020.
P.S. acknowledges the SCIEX Programme, the work has been conducted within the SCIEX-NMS Fellowship, project 11.263, and thanks F. Bagarello for showing and discussing his results \cite{Bagarello-2010-374} during the research visit in Palermo in October 2013. 
J.V. is grateful for the support of ANR NOSEVOL (Project number: ANR 2011 BS01019 01).

%
{\footnotesize
\bibliographystyle{abbrv}
\bibliography{bib}

\begin{thebibliography}{10}
\bibitem{nsa-rem}
Recall that ``self-adjointness'' of a densely defined operator~$H$, \ie\
  $H=H^*$, is a much stronger property than being just ``Hermitian'' or
  ``symmetric'', \ie\ $H \subset H^*$. On the other hand, the terms
  ``non-Hermitian'' and ``non-self-adjoint'' can be read as synonyms in the
  present paper, since we are concerned with a ``fundamental
  non-self-adjointness'' caused for instance by complex coefficients of
  differential operators; that is, we are not interested in closed symmetric
  non-self-adjoint operators here.
\end{thebibliography}
}
\end{document}